\documentclass{article}

\usepackage[10pt]{extsizes}
\usepackage{cmap}
\usepackage[utf8]{inputenc}
\usepackage[T2A]{fontenc}
\usepackage[english]{babel}
\usepackage{amsmath}
\usepackage{amsthm}
\usepackage{amssymb}
\usepackage{amsfonts}
\usepackage{euscript}
\usepackage{enumerate}
\usepackage{hyperref}
\usepackage[affil-it]{authblk}

\usepackage[all]{xy}
\usepackage{epigraph}

\newcommand{\RN}{\mathbb{R}} 
\newcommand{\CN}{\mathbb{C}} 

\newcommand{\eps}{\ensuremath\varepsilon}

\newcommand{\suml}{\sum\limits}
\newcommand{\prodl}{\prod\limits}
\newcommand{\ovl}{\overline}

\newcommand{\gr}{\mathop{\mathrm{gr}}\nolimits}

\newcommand{\SL}{\mathop{\mathrm{SL}}\nolimits}

\newcommand{\Hom}{\mathop{\mathrm{Hom}}\nolimits}

\newcommand{\GL}{\mathop{\mathrm{GL}}\nolimits}

\newcommand{\id}{\ensuremath\operatorname{id}}
\newcommand{\Dic}{\mathbb{D}}
\newcommand{\mm}{\ensuremath\mathfrak{m}}

\newcommand{\Spec}{\operatorname{Spec}}
\newcommand{\Span}{\operatorname{Span}}

\newcommand{\ddx}[2]{\frac{\partial #1}{\partial #2}}

\newcommand{\mf}{\mathfrak}
\newcommand{\Der}{\operatorname{Der}}

\newcommand{\mc}[1]{\mathcal{#1}}
\newcommand{\Tor}{\operatorname{Tor}}
\newcommand{\Specm}{\operatorname{Specm}}
\author{Daniil Klyuev}
\title{Deformations of pairs of Kleinian singularities}
\begin{document}
\newtheorem{thr}{Theorem}[section]
\newtheorem*{thr*}{Theorem}
\newtheorem{lem}[thr]{Lemma}
\newtheorem*{lem*}{Lemma}
\newtheorem{cor}[thr]{Corollary}
\newtheorem{prop}[thr]{Proposition}
\newtheorem{stat}[thr]{Statement}
\newtheorem*{stat*}{Statement}
\newtheorem{example}[thr]{Example}
\theoremstyle{definition}
\newtheorem{defn}[thr]{Definition}
\theoremstyle{remark}
\newtheorem{rem}[thr]{Remark}
\newtheorem*{rem*}{Remark}
\maketitle
\begin{abstract}
Kleinian singularities, i.e., the varieties corresponding to the algebras of invariants of Kleinian groups are of fundamental importance for Algebraic geometry, Representation theory and Singularity theory. The  filtered deformations of these algebras of invariants were classified by Slodowy (the commutative case) and Losev  (the general case). To an inclusion of Kleinian groups, there is the corresponding inclusion of  algebras of invariants. We  classify deformations of these inclusions when the smaller subgroup is normal  in  the larger. 
\end{abstract}
\tableofcontents
\section{Introduction}
A Kleinian singularity is an affine variety of the form $\Spec\CN[u,v]^G$, where $G$ is a finite subgroup of $\SL(2,\CN)$. Kleinian singularities appear in many areas of geometry, algebraic geometry, singularity theory and group theory. Since the action of $G$ does not change the degree of a homogeneous polynomial, $\CN[u,v]^G$ is a graded algebra.

Let us define the notion of a filtered deformation of a graded algebra. 

All algebras are supposed to be associative unital $\CN$-algebras.

\begin{defn}
Suppose that $A$ is a graded algebra. A {\it filtered deformation} of $A$ is a pair $(\mc{A},\chi)$, where $\mc{A}$ is a filtered algebra, and $\chi$ is an isomorphism between $\gr \mc{A}$ and $A$.
\end{defn}

Let us say when two deformations are isomorphic.

\begin{defn}
Suppose that $(\mc{A}_1,\chi_1)$, $(\mc{A}_2,\chi_2)$ are two filtered deformations of $A$, $\phi\colon \mc{A}_1\to \mc{A}_2$ is an isomorphism of filtered algebras. We say that $\phi$ is an isomorphism of deformations if $\chi_2\circ\gr\phi=\chi_1$.
\end{defn}

The main example of deformations of Kleinian singularities are Crawley--Boevey---Holland algebras. They were introduced in their work~\cite{CBH}. 

Suppose that $c$ is an element of $Z(\CN[G])$. We will give a definition of the smash product later, see Definition~\ref{SmashProductDef}.
\begin{defn}
Suppose that $G$ is a Kleinian group. It acts on $\CN\langle u,v\rangle$. Denote by $e$ the element $\frac{1}{|G|}\sum_{g\in G} g$. Consider the algebra $\CN\langle u,v\rangle \# G/(uv-vu-c)$. We can view $e$ as an element of this algebra. The algebra $e(\CN\langle u,v\rangle\# G/(uv-vu-c)) e$ is called a CBH algebra with parameter $c$ and is denoted by $\mathcal{O}_c$.
\end{defn}

We see that $\mc{O}_c$ is a unital algebra with unit $e$. It was proved in~\cite{CBH} that $\mc{O}_c$ is a filtered deformation of $\CN[u,v]^G$.

There exists a natural way of identifying $Z(\CN[G])$ with $\CN\times \mf{h}$, where $\mf{h}$ is a Cartan subalgebra of a simple Lie algebra corresponding to a simply-laced Dynkin diagram. It gives a correspondence between Kleinian groups and simply-laced Dynkin diagrams. This correspondence is called the McKay correspondence. Denote by $W$ the corresponding Weyl group. We see that $W$ acts on $Z(\CN[G])$. It was proved in~\cite{CBH} that:

\begin{enumerate}	
\item
Parameters from $\mf{h}$ correspond to commutative deformations.
\item
For every $c\in Z(\CN[G])$, $w\in W$, $\mc{O}_c$ is isomorphic to $\mc{O}_{wc}$.
\end{enumerate} 

\begin{thr}[Crawley--Boevey---Holland~\cite{CBH}, Kronheimer~\cite{Kronheimer}]
Every commutative filtered deformation of $\CN[u,v]^G$ is isomorphic to $\mc{O}_c$ for some $c\in Z(\CN[G])$ and $\mc{O}_c$ is isomorphic to $\mc{O}_{c'}$ if and only if there exists $w\in W$ such that $c'=wc$.
\end{thr}

\begin{thr}[Losev~\cite{Loseu}]
Every filtered deformation of $\CN[u,v]^G$ is isomorphic to $\mc{O}_c$ for some $c\in Z(\CN[G])$ and $\mc{O}_c$ is isomorphic to $\mc{O}_{c'}$ if and only if there exists $w\in W$ such that $c'=wc$.
\end{thr}

Now we move on to our object of study. Suppose that $G_1\subset G_2$ are finite subgroups of $\SL(2,\CN)$. Then $\CN[u,v]^{G_2}$ is a subset of $\CN[u,v]^{G_1}$. The inclusion $\CN[u,v]^{G_2}\subset\CN[u,v]^{G_1}$ is a homomorphism of graded algebras.

\begin{defn}
Suppose that $i\colon A_2\subset A_1$ is an inclusion of graded algebras, $(\mc{A}_1,\chi_1)$ is a filtered deformation of $A_1$, $\mc{A}_2\subset \mc{A}_1$ is an inclusion of filtered algebras. We say that $(\mc{A}_2,\mc{A}_1,\chi_1)$ is a filtered deformation of $i$ if $\chi_1(\gr\mc{A}_2)=A_2$.
\end{defn}

In this paper we classify filtered deformations of $\CN[u,v]^{G_2}\subset\CN[u,v]^{G_1}$ in the case when $G_1$ is normal in $G_2$.

CBH algebras provide an example of deformations of $\CN[u,v]^{G_2}\subset \CN[u,v]^{G_1}$. Suppose that $c$ is an element of $Z(\CN[G_2])\cap Z(\CN[G_1])$. Then $c$ gives two CBH algebras: one is a deformation of $\CN[u,v]^{G_2}$, the other is a deformation of $\CN[u,v]^{G_1}$. Denote them by $\mc{O}_c^2$, $\mc{O}_c^1$.

\begin{prop}
Suppose $G_1\triangleleft G_2$ are finite subgroups of $\SL(2,\CN)$, $c$ is an element of $Z(\CN[G_1])\cap Z(\CN[G_2])$, $\mc{O}_c^1$ and $\mc{O}_c^2$ are $CBH$-algebras for groups $G_1,G_2$ with parameter $c$. Then there exists an embedding of $\mc{O}_c^2$ into $\mc{O}_c^1$. This embedding is a deformation of $\CN[u,v]^{G_2}\subset \CN[u,v]^{G_1}$.
\end{prop}

We will prove this proposition later.

Consider the image of $Z(\CN[G_1])\cap Z(\CN[G_2])$ under the isomorphism \[Z(\CN[G_1])\cong \CN\times\mf{h}.\] Since \[Z(\CN[G_1])\cap Z(\CN[G_2])=Z(\CN[G_1])^{G_2/G_1},\] its image is $\CN\times\mf{h}^{G_2/G_1}$. 

Every automorphism of a Dynkin diagram gives rise to an automorphism of $\mf{h}$. We will prove that $G_2/G_1$ acts on $\mf{h}$ by automorphisms of this form.

Denote the root system in $\mf{h}$ by $\Phi$ and the corresponding Weyl group by $W$. We have another root system in $\mf{h}^{G_2/G_1}$, defined as follows: $\Phi'=\{\sum_{g\in G_2/G_1}g\alpha\mid\alpha\in\Phi\}\setminus\{0\}$. This root system is called a folded root system. The fact that it is indeed a root system is proved, for example, in~\cite{VO}, solution of Problem 4.4.17.

We will prove that Weyl group $H$ of $\Phi'$ is naturally embedded in $W$. Hence $H$ acts on $Z(\CN[G_1])$.

The main results is as follows:

\begin{thr}
Every filtered deformation of $i$ is of the form $\mc{O}^2_c\subset\mc{O}^1_c$, where $c\in Z(\CN[G_1])\cap Z(\CN[G_2])$. Parameters $c$ and $c'$ give isomorphic deformations if and only if there exists $w\in H$ such that $c'=wc$.
\end{thr}

The structure of the paper is as follows. In Sections~\ref{DefSec} and~\ref{InfSec} we define deformations of an algebra and of an inclusion of algebras over a base and recall some technical facts about them. 

We start with commutative case. In section~\ref{DerSec} we prove that each derivation from $\CN[u,v]^{G_2}$ to $\CN[u,v]^{G_1}$ lifts to a derivation of $\CN[u,v]$. In section~\ref{UnivAlgSec} we recall the result of Slodowy on the universal commutative deformation of $\CN[u,v]^G$. In sections~\ref{UniqArrowSec} and~\ref{InjSec} we prove that each commutative deformation $\mc{A}_2\subset\mc{A}_1$ of $\CN[u,v]^{G_2}\subset\CN[u,v]^{G_1}$ is uniquely recovered from $\mc{A}_2$. In section~\ref{NormalCaseSec} we find a universal commutative deformation of $\CN[u,v]^{G_2}\subset\CN[u,v]^{G_1}$ using this result.

Then we deal with noncommutative case. In section~\ref{CBHSec} we recall the definition of a CBH algebra and restate the results of the previous sections in the language of CBH algebras. In section~\ref{NoncommutativeNormalSec} we construct a universal deformation from the universal commutative deformation.

\subsection{Acknowledgments}
I would like to thank Ivan Losev for formulation of the problem, stimulating discussions and for remarks on the previous versions of this paper. I am grateful to Pavel Etingof for his help with rewriting section~\ref{InjSec}. The paper is supported by «Native towns», a social investment program of PJSC «Gazprom Neft»
\section{Definitions and general properties of flat deformations}
\label{DefSec}
Let $R$ be a commutative graded algebra such that $R_0=\CN$, $R_i$ are finite-dimensional and let $\mm=R_{>0}=\suml_{i>0}R_i$ be a maximal ideal of $R$.

\begin{defn}
Let $A$ be a commutative graded algebra. A deformation of $A$ over $R$ is a pair $(\mc{A},\chi)$,where $\mc{A}$ is a graded algebra over $R$, flat as an $R$-module, and $\chi$ is an isomorphism between $\mc{A}/\mm\mc{A}$ and $A$.
\end{defn}
\begin{defn}
Suppose that $(\mc{A},\chi)$ is a deformation of $A$ over $R$, $(\mc{B},\psi)$ is a deformation of $A$ over $S$ and $f$ is a homomorphism of graded algebras from $\mc{A}$ to $\mc{B}$. We say that $f$ is a morphism of deformations if the following holds:
\begin{enumerate}
\item
$f(R)\subset S$
\item
The following triangle is commutative
$$\xymatrix{
\mc{A}/\mc{A}R_{>0}\ar[dr]^{\chi} \ar[r]^{\ovl{f}}& \mc{B}/\mc{B}S_{>0}\ar[d]^{\psi}\\
 & A
}
$$
where $\ovl{f}$ is the homomorphism induced by $f$.
\end{enumerate}
\end{defn}

The notion of a deformation over a base is a generaization of the notion of a filtered deformation:

\begin{defn}
Suppose that $A'=\bigcup_{i=0}^{\infty} A'_{\leq i}$ is a filtered algebra. Its Rees algebra is defined as follows: $\mc{A}=\suml_{i=0}^{\infty} t^i A'_{\leq i}$. It has a structure of $\CN[t]$-algebra.
\end{defn}

We see that the Rees algebra is a free $\CN[t]$-module, $\mc{A}/t\mc{A}\cong \gr A'$ and $\mc{A}/(t-1)\mc{A}\cong A'$. On the other hand, every deformation $\mc{A}$ of $A$ over $\CN[t]$ defines a filtered deformation $\mc{A}/(t-1)\mc{A}$. This construction is an inverse to taking Rees algebra. We conclude that a filtered deformation is the same as a deformation over $\CN[t]$.


Now we define a deformation of a homomorphism of graded algebras.
\begin{defn}
Let $f\colon A\to B$ be a homomorphism of graded algebras. Suppose $\mc{A}$ is a deformation of $A$ over $R$, $\mc{B}$ is a deformation of $B$ over $R$, $F\colon\mc{A}\to\mc{B}$ is an $R$-linear homomorphism of graded algebras. We say that $F$ is a deformation of $f$ if the induced morphism $\ovl{F}\colon \mc{A}/\mc{A}\mm\to \mc{B}/\mc{B}\mm$ coincides with $f$ after identifying $\mc{A}/\mc{A}\mm$ with $A$ and $\mc{B}/\mc{B}\mm$ with $B$.
\end{defn}

\begin{defn}
Let $F_1\colon \mc{A}_1\to\mc{B}_1$ be a deformation of $f$ over $R$, $F_2\colon\mc{A}_2\to\mc{B}_2$ be a deformation of $f$ over $S$. Suppose $g\colon\mc{A}_1\to\mc{A}_2$ is a morphism of deformations of $A$, $h\colon\mc{B}_1\to\mc{B}_2$ is a morphism of deformations of $B$. We say that $(g,h)$ is a morphism of deformations of $f$ if the following square commutes
$$\xymatrix{
\mc{A}_1\ar[d]_{g} \ar[r]^{F_1}& \mc{B}_1 \ar[d]^{h}\\
\mc{A}_2 \ar[r]_{F_2} & \mc{B}_2}
$$
\end{defn}

The Rees construction gives a correspondence between filtered deformations of $f\colon A_1\to A_2$ and deformations of $f$ over $\CN[t]$.


Now we move for technical statements we will need later.
\begin{lem}
\label{DeformationOfIdLem}
Suppose $A$, $R$ are graded commutative algebras and $\mc{A}$, $\mc{A}_1$, $\mc{A}_2$ are deformations of $A$ over $R$.
\begin{enumerate}
\item
Let $a_i$ be homogeneous elements of $\mc{A}$ such that their images form a basis in $A$. Then $a_i$ form a basis over $R$ in $\mc{A}$.
\item
 Suppose that $\phi\colon\mc{A}_1\to\mc{A}_2$ is a deformation of $\id_A$. In other words, $\phi$ is an $R$-linear homomorphism of deformations of $A$. Then $\phi$ is an isomorphism of deformations of $A$.
\end{enumerate} 
\end{lem}
\begin{proof}
The first statement is standard. 

To prove the second we choose any $R$-basis $\{a_i\}$ of $\mc{A}_1$ provided by the first statement. We note that the images of $\phi(a_i)$ in $A$ equal to the images of $a_i$ on $A$. Using the first statement again we deduce that $\phi(a_i)$ is a $R$-basis of $\mc{A}_2$. Hence $\phi$ is invertible. It follows from definitions that $\phi^{-1}$ is also a morphism of deformations of $A$.
\end{proof}
\begin{lem}
\label{DeformationIndeedLemma}
Let $x_i$ be a homogeneous variable of positive degree, $f$ be a homogeneous element of $\CN[x_1,\ldots,x_n]$, $F$ be a homogeneous element of $R[x_1,\ldots,x_n]$ such that $F-f\in\mm[x_1,\ldots,x_n]$. Then $\mc{A}=R[x_1,\ldots,x_n]/(F)$ is a free $R$-module.
\end{lem}
\begin{proof}
Since $\mc{A}=R[x_1,\ldots,x_n]/(F)$ is a graded $R$-module it is enough to check that $\Tor_1^R(\mc{A},\CN)=0$. The $R$-module $\mc{A}$ has a free resolution $0\to R[x_1,\ldots,x_n]\to R[x_1,\ldots,x_n]\to \mc{A}\to 0$, where the first map is multiplication by $F$. Taking tensor product of this resolution with $\CN$ over $R$ we get $\CN[x_1,\ldots,x_n]\to \CN[x_1,\ldots,x_n]$, where the map is multiplication by $f$. Since multiplication by $f$ is injective we deduce that $\Tor_1^R (\mc{A},\CN)=0$.

\end{proof}

\begin{cor}
Let $A=\CN[x_1,\ldots,x_n]/(f)$, where $x_i$ are homogeneous variables, $f$ is a homogeneous polynomial in $x_i$. Let $\mc{A}$ be a commutative deformation of $A$ over $R$. Then there exists a homogeneous polynomial $F\in f+\mm[x_1,\ldots,x_n]\subset R[x_1,\ldots,x_n]$ of degree equal to the degree of $f$ such that $\mc{A}\cong R[x_1,\ldots,x_n]/(F)$. Moreover, this is an isomorphism of deformations of $A$, where the structure of a deformation of $A$ on $R[x_1,\ldots,x_n]/(F)$ is given by Lemma~\ref{DeformationIndeedLemma}.
\end{cor}
\begin{proof}
Let $X_i$ be any homogeneous lift of $x_i\in A$ to $\mc{A}$. Consider a homomorphism of $R$-algebras $\psi\colon R[x_1,\ldots,x_n]\to \mc{A}$, $\psi(x_i)=X_i$. Using graded Nakayama lemma we see that $\psi$ is surjective. Since $f=0$ in $A$ we have $\psi(f)\in\mm\mc{A}$. Since $\psi$ is surjective we can find $H\in\mm[x_1,\ldots,x_n]$ with $\deg H=\deg f$ such that $\psi(H)=\psi(f)$. It follows that $F=f-H$ belongs to the kernel of $\psi$. 

Therefore we obtain from $\psi$ a homomorphism $\phi\colon R[x_1,\ldots,x_n]/(F)\to\mc{A}$, $\phi(x_i)=X_i$. We see that $\phi$ is an $R$-linear morphism of deformations of $A$. It follows from Lemma~\ref{DeformationOfIdLem} that $\phi$ is an isomorphism of deformations of $A$.
\end{proof}

Now we formulate several statements for future use.
\begin{stat}
\label{BaseTensoringStat}
\begin{enumerate}
\item
Suppose that $\mc{A}$ is a deformation of $A$ over $R$, $\phi\colon R\to S$ is a homomorphism of graded algebras. Then $\mc{A}\otimes_R S$ is a deformation of $A$ over $S$. Moreover, the natural homomorphism from $\mc{A}$ to $\mc{A}\otimes_R S$ is a morphism of deformations of $A$.
\item
Suppose that $F\colon\mc{A}\to\mc{B}$ is a deformation of $f\colon A\to B$ over $R$, $\phi\colon R\to S$ is a homomorphism of graded algebras. Then $F\otimes\id\colon \mc{A}\otimes_R S\to \mc{B}\otimes_R S$ is a deformation of $f$ over $S$. Moreover, the pair of natural homomorphism $\mc{A}\to\mc{A}\otimes_R S$, $\mc{B}\to\mc{B}\otimes_R S$ is a morphism of deformations of $f$.
\end{enumerate}

\end{stat}
\begin{stat}
\label{EveryMorphismIsABaseChangeStat}
\begin{enumerate}
\item
Suppose that $\mc{A}$, $\mc{B}$ are deformations of $A$ over $R$, $S$ respectively, $\phi\colon\mc{A}\to\mc{B}$ is a morphism of deformations. $\phi|_{R}$ gives a structure of $R$-module on $S$. Consider the natural homomorphism $g\colon\mc{A}\otimes_{R} S\to\mc{B}$. Then $g$ is an isomorphism of deformations. Moreover, the composition $\mc{A}\to\mc{A}\otimes_{R} S\to \mc{B}$ coincides with $\phi$.
\item
Suppose that $F_1\colon\mc{A}_1\to\mc{B}_1$, $F_2\colon\mc{A}_2\to\mc{B}_2$ are deformations of $f\colon A\to B$ over $R$, $S$ respectively. Suppose that $\phi\colon \mc{A}_1\to\mc{A}_2$, $\psi\colon\mc{B}_1\to\mc{B}_2$ is a morphism of deformations of $f$. The map $\phi|_{R}=\psi|_{R}$ gives a structire of $R$-module on $S$. Then the pair of natural homomorphisms $\mc{A}_1\otimes_{R} S\to\mc{A}_2$, $\mc{B}_1\otimes_{R} S\to\mc{B}_2$ is an isomorphism of deformations of $f$. Analogous statement about composition holds.
\end{enumerate}
\end{stat}
\begin{cor}
\label{EqualRestrictionsOnUniversalBaseCor}
Suppose that $\mc{B}_1$ and $\mc{B}_2$ are deformations of $A$ over $S$, $\mc{A}$ is a deformation of $A$ over $R$, $\phi\colon \mc{A}\to\mc{B}_1$, $\psi\colon\mc{A}\to\mc{B}_2$ are morphisms of deformations. Suppose that $\phi(r)=\psi(r)\in S$ for all $r\in R$. Then $\mc{B}_1$ is isomorphic to $\mc{B}_2$ as a deformation of $A$.
\end{cor}
\begin{cor}
\label{EqualIffRestictionsAreEqualCor}
Suppose that $\mc{A},\mc{A}'$ are deformations of $A$ over $R$, $S$. Two morphisms of deformations $\phi,\psi$ from $\mc{A}$ to $\mc{A}'$ are equal if and only if $\phi|_R$ is equal to $\psi|_R$.
\end{cor}


\section{The uniqueness of $F$ when $\mc{A}$, $\mc{B}$ are fixed.}
\label{InfSec}
In sections~\ref{InfSec}-\ref{NormalCaseSec} we consider only commutative deformations.


\begin{lem}
\label{LemEpsAIsWellDefined}
Suppose that $R$ is finite-dimensional and there exists nonzero homogeneous $\eps\in R$ such that $\eps\mm=0$, $a$ is an element of $A$, $\tilde{a}$ is a lift of $a$ to $\mc{A}$. Then $\eps\tilde{a}$ does not depend on $a$ and $a\mapsto \eps\tilde{a}$ is a bijection from $A$ to $\eps\mc{A}$.
\end{lem}
\begin{proof}
Any two lifts of $a$ differ by an element of $\mm\mc{A}$. This proves the first statement. The second statement follows from the fact that $\mc{A}$ is a free $R$-module.
\end{proof}

So for $a\in A$ we can define $\eps a\in \eps\mc{A}$ and for $a\in \eps\mc{A}$ we can define $a/\eps\in A$.


\begin{lem}
\label{ElementsToEpsLemma}
Suppose that $(a_1,\ldots,a_n)$ and $(b_1,\ldots,b_n)$ are non-equal ordered sets of elements of $B$ such that $a_i+\mm=b_i+\mm$. Then there exist homogeneous ideals $I\subset J$ of $B$ such that
\begin{enumerate}
\item
$(a_1+I,\ldots,a_n+I)\neq (b_1+I,\ldots, b_n+I)$
\item
$(a_1+J,\ldots,a_n+J)=(b_1+J,\ldots,b_n+J)$
\item
The kernel of projection $B/I \to B/J$ is one-dimensional. 
\end{enumerate}
\end{lem}
\begin{proof}
Let $d$ be the maximal positive integer such that $(a_1+B^{\geq d},\ldots, a_n+B^{\geq d})=(b_1+B^{\geq d},\ldots,b_n+B^{\geq d})$. Now is easy to find such a pair $I\subset J$ with $J=B^{\geq d}$.
\end{proof}

\begin{prop}
\label{PropTwoDeformationsThenDerivation}
Let $F_1,F_2\colon\mc{A}\to\mc{B}$ be two deformations of $f\colon A\to B$ over $R$ such that for all $x\in \mc{A}$ we have $F_1(x)-F_2(x)\in \eps\mc{B}$. Then there exists a unique map $d\colon A\to B$ such that $\eps d(x)=(F_1-F_2)(\tilde{x})$, where $\tilde{x}$ is any lift of $x$ to $\mc{A}$. Moreover, $d$ is a homogeneous derivation of degree $-\deg\eps$.
\end{prop}
\begin{proof}
Consider the map $D=F_1-F_2$. This map is $R$-linear and satisfies $D(a)D(b)=0$ for any $a,b\in \mc{A}_1$. Then \begin{equation}
\label{EqBigDIsDerivation}
D(ab)=F_1(a)D(b)+D(a)F_2(b)=F_1(a)D(b)+D(a)F_1(b)=F_2(a)D(b)+D(a)F_2(b)
\end{equation} for any $a,b\in\mc{A}_1$.

We see that $D(\mm\mc{A})=0$, so $d(x)=D(\tilde{x})/\eps$ is well-defined.

It follows from~\eqref{EqBigDIsDerivation} that $d$ is a derivation. Since $F_1$ and $F_2$ are homogeneous maps the degree of $d$ equals to $-\deg\eps$.
\end{proof}
\begin{cor}
\label{CorNonUniqenessArrowDeformation}
Let $F_1,F_2$ be two distinct deformations of $f\colon A_1\to A_2$ over $B$. Then there exists a non-zero derivation of $A_1$ into $A_2$ of negative degree.
\end{cor}
The corollary easily follows from Proposition~\ref{PropTwoDeformationsThenDerivation} and Lemma~\ref{ElementsToEpsLemma}.

The following theorem follows from Theorem 2.4 in~\cite{Slodowy}.
\begin{thr}
\label{ExistenceOfMorphismOfDeformationsThr}
Let $p$ be a homogeneous element of $\CN[x_1,\ldots,x_n]$. Denote $\CN[x_1,\ldots,x_n]/(p)$ by $A$. Suppose that there exist homogeneous elements \[u_1,\ldots,u_m\in\CN[x_1,\ldots,x_n]\] of degree less than $\deg p$ such that their images in $$\CN[x_1,\ldots,x_n]/(\ddx{p}{x_1},\ddx{p}{x_2},\ldots,\ddx{p}{x_n})$$ form a basis. Suppose that $R=\CN[y_1,\ldots,y_m]$, \[\mc{A}_0=R[x_1,\ldots,x_n]/(P(x_1,\ldots,x_n,y_1,\ldots,y_m)),\] where the degree of $y_i$ equals $\deg f-\deg u_i$, \[P(x_1,\ldots,x_n,y_1,\ldots,y_m)=p(x_1,\ldots,x_n)-\sum_{i=1}^m u_i(x_1,\ldots,x_n)y_i.\] Then $\mc{A}_0$ is a universal commutative deformation of $A$. In other words, any other commutative deformation is obtained via a unique base change from $\mc{A}_0$.
\end{thr}

\section{Structure of $\Der_{\CN}(\CN[u,v]^{G_2},\CN[u,v]^{G_1})$}
\label{DerSec}
In this section $G_1\subset G_2$ are finite subgroups of $\SL(2,\CN)$.

Define a map $r\colon\Der_{\CN}(\CN[u,v],\CN[u,v])\to \Der_{\CN}(\CN[u,v]^{G_2},\CN[u,v])$ as follows: $r(D)=D|_{\CN[u,v]^{G_2}}$. We see that $r$ preserves degrees. 

We are going to prove the next theorem:
\begin{thr}
\label{MainDerThr}
\begin{enumerate}
\item
$r$ is a bijection.
\item
Suppose that $G_1\subset G_2$. Then $r^{-1}(\Der_{\CN}(\CN[u,v]^{G_2},\CN[u,v]^{G_1}))$ consists of all $G_1$-equivariant derivations of $\CN[u,v]$.
\end{enumerate}
\end{thr}
The theorem is proved below in this section.
\begin{cor}
\label{NoDerivationOfHegativeDegCor}
Suppose that $G_1\subset G_2$ are non-trivial finite subgroups of $\SL(2,\CN)$. Then there are no non-zero homogeneous derivations of $\CN[u,v]^{G_2}$ into $\CN[u,v]^{G_1}$ of negative degree 
\end{cor}
\begin{proof}
Using the theorem we can reformulate the statement as follows: there are no non-zero homogeneous $G_1$-equivariant derivations of $\CN[u,v]$ of negative degree. Assume the converse. Chose any nonzero homoheneous $G_1$-equivariant derivation of $\CN[u,v]$ of negative degree. Restricting it to $(\CN[u,v])_1=\Span(u,v)$ we get a nonzero operator $D\colon \Span(u,v)\to \CN$ intertwining action of $G_1$. The space $\CN$ is a trivial representation of $G_1$, the space $\Span(u,v)$ is a tautological representation of $G_1$. There is no trivial representation inside tautological, so $D=0$.
\end{proof}
Combining this with Corollary~\ref{CorNonUniqenessArrowDeformation} we get the next
\begin{cor}
\label{UniqueArrowDeformationCor}
Suppose $i\colon \CN[u,v]^{G_1}\to \CN[u,v]^{G_2}$ is a homomorphism of graded algebras, $\mc{A}_1$, $\mc{A}_2$ are deformations of $\CN[x,y]^{G_1}$, $\CN[x,y]^{G_2}$ over $B$. Suppose that $F_1,F_2\colon \mc{A}_1\to \mc{A}_2$ are deformations of $i$ over $B$. Then $F_1=F_2$.
\end{cor}

Suppose that $X$ is a smooth affine variety and a finite group $G$ acts on $X$ algebraically. The following fact follows from Proposition 4.11 in~\cite{Drezet}.
\begin{stat}
\label{CategoricalQuotientStat}
Denote $\Spec \CN[X]^G$ by $X/G$. Let $\pi\colon X\to X/G$ be the quotient morphism of algebraic varieties corresponding to inclusion $\CN[X]^G\subset \CN[X]$. Then the following holds
\begin{enumerate}
\item
$\pi$ is finite.
\item
Each fiber of $\pi$ is a single orbit of action of $G$.
\item
$Y$ is smooth in the points corresponding to free orbits of $G$.
\item
$\pi$ is \'etale in the points with trivial stabilizer.
\end{enumerate} 
\end{stat}

Suppose that $\phi\colon X\to Y$ is a morphism of algebraic varieties, $D$ is an element of $\Der(\CN[X])$.  Then $D\circ\phi^*$ belongs to $\Der(\CN[Y],\CN[X])$. So we have a mapping from $\Der(\CN[X])$ to $\Der(\CN[Y],\CN[X])$. Denote it by $\Phi$.
\begin{prop}
\label{PropDerivationsAlmostEtale}
Suppose that $X,Y$ are irreducible affine algebraic varieties, $X$ is smooth, $\phi\colon X\to Y$ is a finite dominant morphism. Suppose that there exists a codimension two subvariety $Z$ of $Y$ such that
\begin{enumerate}
\item
$Y\setminus Z$ is smooth.
\item
$\phi|_{X\setminus\phi^{-1}(Z)}$ is \'etale.
\end{enumerate}
Then $\Phi\colon \Der(\CN[X])\to \Der(\CN[Y],\CN[X])$ is a bijection.
\end{prop}
\begin{proof}
Suppose that $Y\setminus Z=\bigcup\limits_{i=1}^n Y_i$, where $Y_i$ are open affine subsets of $Y$. Denote $\phi^{-1}(Y_i)$ by $X_i$. Then $Y_i$ is smooth and $\phi|_{X_i}$ is \'etale for all $i$ from $1$ to $n$.

We will need the following lemma.

\begin{lem*}
Suppose that the same conditions hold. Suppose that $D_i$ are elements of $\Der(\CN[Y_i],\CN[X_i])$ such that $D_i|_{Y_i\cap Y_j}=D_j|_{Y_i\cap Y_j}$ for all $i,j$ from $1$ to $n$. Then there exists a unique $D\in\Der(\CN[Y],\CN[X])$ such that $D|_{Y_i}=D_i$.
\end{lem*}
\begin{proof}
Let $f\in \CN[Y]$. We should have $D(f)=D_i(f)$ for all $i=1,2,\ldots, n$. Since $f\in \CN[Y_i\cap Y_j]$ we have $D_i(f)=D_j(f)$ for all $i,j$. So $D(f)=D_i(f)$ is a well-defined derivation from $\CN[Y]$ to $\CN(X)$.

It remains to check that $D(f)$ indeed belongs to $\CN[X]$. Since $D(f)=D_i(f)$ for all $i$ function $D(f)$ is regular on $\bigcup_{i=1}^n X_i$. It follows from Hartog's theorem that $D(f)$ belongs to $\CN[X]$.
\end{proof}

Define $\Phi_i\colon\Der(\CN[X_i])\to\Der(\CN[Y_i],\CN[X_i])$ in the same way as $\Phi$. Using lemma we see that bijectivity of $\Phi$ follows from bijectivity of $\Phi_i$.

So we can assume that $X, Y$ are affine, smooth and $\phi$ is \'etale. It follows that $h\colon\CN[X]\otimes_{\CN[Y]}\Omega_{\CN[Y]/\CN} \to\Omega_{\CN[X]/\CN}$, $h(c\otimes db)=c d\phi^*(b)$, is an isomorphim of $\CN[X]$-modules. 
Applying $\Hom_{\CN[X]}(-,\CN[X])$ we obtain a bijection \[h^*\colon\Hom_{\CN[X]}(\Omega_{\CN[X]/\CN},\CN[X])\xrightarrow{\sim} \Hom_{\CN[X]}(\CN[X]\otimes_{\CN[Y]}\Omega_{\CN[Y]/\CN},\CN[X])\] We see that $\Hom_{\CN[X]}(\Omega_{\CN[X]/\CN},\CN[X])$ is isomorphic to $\Der(\CN[X],\CN[X])$ and $\Hom_{\CN[X]}(\CN[X]\otimes_{\CN[Y]}\Omega_{\CN[Y]/\CN},\CN[X])$ is isomorphic to $\Der(\CN[Y],\CN[X])$. It is not hard to prove that the following diagram, where the top arrow is $h^*$, commutes:
$$ \xymatrix{
\Hom_{\CN[X]}(\Omega_{\CN[X]/\CN},\CN[X])\ar[d]\ar[r] & \Hom_{\CN[X]}(\CN[X]\otimes_{\CN[Y]}\Omega_{\CN[Y]/\CN},\CN[X])\ar[d]\\
\Der(\CN[X],\CN[X])\ar[r]^{\Phi} & \Der(\CN[Y],\CN[X])
}$$
Hence $\Phi$ is an isomorphism.
\end{proof}

\begin{proof}[Proof of Theorem~\ref{MainDerThr}]
Let $X=\CN^2$. It follows from Statement~\ref{CategoricalQuotientStat} that the quotient morphism $X\to X/G$ satisfies conditions of Proposition~\ref{PropDerivationsAlmostEtale}. The first part of theorem follows. 

To prove the second part we note that if $D$ is a derivation of $\CN[u,v]$ such that $D|_{\CN[u,v]^{G_2}}\in \Der(\CN[u,v]^{G_2},\CN[u,v]^{G_1})$, then \[(\frac{1}{|G_1|}\sum_{g\in G_1} gDg^{-1})|_{\CN[u,v]^{G_2}}=D|_{\CN[u,v]^{G_2}}.\] The map $\frac{1}{|G_1|}\sum_{g\in G_1} gDg^{-1}$ is also a derivation. Since $r$ is a bijection we deduce that $D=\frac{1}{|G_1|}\sum_{g\in G_1} gDg^{-1}$. It follows that $D$ is $G_1$-equivariant.
\end{proof}
\section{Universal deformations of Kleinian singulatities}
\label{UnivAlgSec}
Suppose that $G$ is a finite subgroup of $\SL(2,\CN)$. We want to formulate several properties of the universal commutative deformation of $\CN[u,v]^G$ for future use. The classification of universal deformations of Kleinian singularities is a result of Slodowy~\cite{Slodowy}. It is well-known (see~\cite{VinbergPopov}, subsection 0.13 or~\cite{Burban}, for example) that $\CN[u,v]^G\cong \CN[x,y,z]/f(x,y,z)$, where all possible combinations of $G,f,\deg x,\deg y,\deg z$ are as follows:
\begin{enumerate}
\item
$G=C_n$, $f=x^n+yz$, $\deg x=2$, $\deg y=n$, $\deg z=n$
\item
$G=\mathbb{D}_n$, $f=xy^2+z^2+x^{n+1}$, $\deg x=4$, $\deg y=2n$, $\deg z=2n+2$
\item
$G=\mathbb{T}$, $f=x^4+y^3+z^2$, $\deg x=6$, $\deg y=8$, $\deg z=12$
\item
$G=\mathbb{O}$, $f=x^3y+y^3+z^2$, $\deg x=8$, $\deg y=12$, $\deg z=18$
\item
$G=\mathbb{I}$, $f=x^5+y^3+z^2$, $\deg x=12$, $\deg y=20$, $\deg z=30$
\end{enumerate} 

\begin{defn}
Suppose that $M$ is a module over a ring $R$, $M'$ is a submodule of $M$. If every nonzero submodule of $M$ has nonzero intersection with $M'$, we call $M'$ an essential submodule.
\end{defn}

\begin{stat}
\label{StatSocle}
The quotient $\CN[x,y,z]/(\ddx{f}{x},\ddx{f}{y},\ddx{f}{z})$ has a simple socle. In other words, there exists an element $a_M$ of $\CN[x,y,z]/(\ddx{f}{x},\ddx{f}{y},\ddx{f}{z})$ such that $\CN a_M$ is an essential submodule of $\CN[x,y,z]/(\ddx{f}{x},\ddx{f}{y},\ddx{f}{z})$.
\end{stat}
\begin{proof}
Write down all possible $f$: $x^n+yz$, $xy^2+z^2+x^{n+1}$, $x^4+y^3+z^2$, $x^3y+y^3+z^2$, $x^5+y^3+z^2$. It is easy to check that the following elements have the desired property: $x^{n-2}$, $x^n$, $x^2y$, $x^4$, $x^3y$.
\end{proof}
\begin{rem}
We see that $\deg a_M=\deg f-4$.
\end{rem}
Let $u_1,\ldots,u_m$ be homogeneous elements of $\CN[x,y,z]$ such that their images in $\CN[x,y,z]/(\ddx{f}{x},\ddx fy,\ddx fz)$ form a linear basis. Suppose that $R_0=\CN[y_1,\ldots,y_m]$, $\mc{A}_0=R_0[x,y,z]/(f(x,y,z)-\sum_{j=1}^m y_ju_j)$.
It follows that $\mc{A}_0$ is a deformation of $\CN[u,v]^G$ over $R_0$ satisfying the conditions of Theorem~\ref{ExistenceOfMorphismOfDeformationsThr}, so $\mc{A}_0$ is a universal commutative deformation of $\CN[u,v]^G$.

\begin{lem}
\label{DeterminedByBLem}
Let $\mc{A}$ be a deformation of $\CN[u,v]^G\cong \CN[x,y,z]/(f)$ over $R$. Then there exist unique $r_1,\ldots,r_m$ such that $\mc{A}$ is isomorphic to $R[x,y,z]/(f+\sum_{j=1}^m r_iu_i)$ as a deformation.
\end{lem}
The proof is straightforward.

\begin{defn}
The previous lemma gives us a surjection $\pi\colon R[x,y,z]\to \mc{A}$. We will call this surjection {\it canonical}.
\end{defn}


\section{The uniqueness of the bigger deformation}
\label{UniqArrowSec}
Let $G_1\subset G_2$ be finite subgroups of $\SL(2,\CN)$. The following theorem is the main step in classifying commutative deformations.
\begin{thr}
\label{MainTheorem}
Suppose that $R$ is a graded commutative algebra, $\mc{B}$ is a deformation of $\CN[u,v]^{G_2}$ over $R$, $F_1\colon \mc{B}\to \mc{A}_1,F_2\colon \mc{B}\to\mc{A}_2$ are two deformations of $i\colon\CN[u,v]^{G_2}\to\CN[u,v]^{G_1}$ over $R$. Then there exists an isomorphism of deformations $g\colon \mc{A}_1\to\mc{A}_2$ such that $F_2=gF_1$.
\end{thr}
\begin{rem*}
We see that $g$ is $R$-linear, so it is a deformation of $\id_{\CN[u,v]^{G_1}}$. Using Corollary~\ref{UniqueArrowDeformationCor} we see that $g$ is unique.
\end{rem*}

It is enough to prove that $\mc{A}_1$ is isomorphic to $\mc{A}_2$ as a deformation of $\CN[u,v]^{G_1}$. The equality $F_2=gF_1$ will follow from Corollary~\ref{UniqueArrowDeformationCor}. We assume that this is not the case: $\mc{A}_1$ is not isomorphic to $\mc{A}_2$.

The proof will be in two steps. In this section we prove that Proposition~\ref{PropForMainLemma} implies Theorem~\ref{MainTheorem}. In the next section we prove Proposition~\ref{PropForMainLemma}.

Let $f\colon \mc{B}\to \mc{A}$ be a deformation of $i\colon \CN[u,v]^{G_2}\subset \CN[u,v]^{G_1}$ over $R$. Consider the canonical surjections $\pi\colon R[x_1,y_1,z_1]\to \mc{A}\cong R[x_1,y_1,z_1]/(P)$, $\tau\colon R[x_2,y_2,z_2]\to \mc{B}\cong B[x_2,y_2,z_2]/(T)$. Denote by $\phi$ any homomorphism of $R$-algebras from $R[x_2,y_2,z_2]$ to $R[x_1,y_1,z_1]$ such that $\pi_1\phi=f\pi_2$. Applying both sides to $T$ we see that there exists $Q\in R[x_1,y_1,z_1]$ such that $\phi(T)=PQ$. Since $P$ is not a zero divisor in $R[x_1,y_1,z_1]$,  $Q$ is unique.

Let $\mc{A}_1=R[x_1,x_1,z_1]/(P_1)$, $\mc{A}_2=R[x_1,y_1,z_1]/(P_2)$, $\mc{B}=R[x_2,y_2,z_2]/(T)$.
\begin{lem}
There exist homogeneous ideals $I,J$ of $R$ such that
\begin{enumerate}
\item
$I\subset J$
\item
$\mc{A}_1/I\mc{A}_1$ is not isomorphic to $\mc{A}_2/I\mc{A}_2$.
\item
$\mc{A}_1/J\mc{A}_1$ is isomorphic to $\mc{A}_2/J\mc{A}_2$.
\item
The kernel of projection $R/I\twoheadrightarrow R/J$ is one-dimensional.
\end{enumerate}
\end{lem}
\begin{proof}
Using Lemma~\ref{DeterminedByBLem} we can reformulate the second and the third claim as follows:
\begin{enumerate}
\item
$P_1+I[x,y,z]\neq P_2+I[x,y,z]$
\item
$P_1+J[x,y,z]=P_2+J[x,y,z]$
\end{enumerate}
We get the result from Lemma~\ref{ElementsToEpsLemma}.
\end{proof}

Replace $R$ with $R/I$. Now we can assume that there exists an element $\eps\in R$ such that $\eps\mm=0$ and $\mc{A}_1/\eps\mc{A}_1\cong \mc{A}_2/\eps\mc{A}_2$. Let $S=R/(\eps)$.

Using Corollary~\ref{UniqueArrowDeformationCor} we see that morphisms $F_1\otimes_R S$ and $F_2\otimes_R S$ coincide after we identify $\mc{A}_1/\eps\mc{A}_1$ with $\mc{A}_2/\eps\mc{A}_2$. 

Denote projections from $R[x_1,y_1,z_1]$ to $\mc{A}_1,\mc{A}_2$ by $\pi_1,\pi_2$. We have $\pi_1\otimes_R S=\pi_2\otimes_R S$. Denote the projection from $R[x_2,y_2,z_2]$ to $\mc{B}$ by $\tau$. Lift $F_1\otimes_R S=F_2\otimes_R S$ to $\phi\colon S[x_2,y_2,z_2]\to S[x_1,y_1,z_1]$. In other words, \[ (F_i\otimes_R S)\circ (\tau\otimes_R S)=(\pi_i\otimes_R S)\circ \phi\] for $i=1,2$. Now for $i=1,2$ we can find $\phi_i$ such that $F_i\circ\tau=\pi_i\circ\phi_i$ and $\phi_i\otimes_R S=\phi$. 

Let $Q_i=\frac{\phi_i(T)}{P_i}$. We note that the images of $Q_1$ and $Q_2$ in $S[x_1,y_1,z_1]$ coincide. We write $Q_2=Q_1+\eps\Delta_Q$, $P_2=P_1+\eps\Delta_P$, $\phi_2(x_2)=\phi_1(x_2)+\eps\delta_x$, similarly for $y_2,z_2$. Using Lemma~\ref{LemEpsAIsWellDefined} we may assume that $\Delta_Q,\Delta_P,\delta_x,\delta_y,\delta_z\in\CN[x_1,y_1,z_1]$.

We have \[Q_2P_2-\phi_2(T)=Q_1P_1+\eps (\Delta_Q p+\Delta_P q)-\phi_1(T)-\eps(\delta_x t'_x+\delta_y t'_y+\delta_z t'_z),\] where $p,q$ are the images of $P_i,Q_i$ in $\CN[x_1,y_1,z_1]$ and $t$ is the image of $T$ in $\CN[x_2,y_2,z_2]$. Since $Q_iP_i-\phi_i(T)=0$ we get \[\Delta_Q p+\Delta_P q=\delta_x t'_x+\delta_y t'_y+\delta_z t'_z.\]

Since $P_1$ is not equal to $P_2$, $\Delta_P$ is not equal to zero. We also have $P_i=p+\sum_{j=1}^m r^i_j u_j$, where $\CN[u,v]^{G_1}=\CN[x_1,y_1,z_1]/(p)$, $u_1,\ldots,u_m$ is a basis of $\CN[x_1,y_1,z_1]/(p'_x,p'_y,p'_z)$ and $r^i_j\in R$. Therefore the image of $\Delta_P$ in $\CN[x_1,y_1,z_1]/(p'_x,p'_y,p'_z)$ is nonzero. In order to get a contradiction we will prove that $\Delta_P q$ does not belong to the ideal $(p,t'_x,t'_y,t'_z)$. It is enough to prove that the image of $\Delta_P q$ in $\CN[u,v]^{G_1}$ does not belong to the ideal $(t'_x,t'_y,t'_z)$:

\begin{prop}
\label{PropForMainLemma} 
Suppose that $G_1\subset G_2$ are finite subgroups of $\SL(2,\CN)$, $\pi_i\colon \CN[x_i,y_i,z_i]\to \CN[u,v]^{G_i}$ are canonical projections and that the kernel of $\pi_i$ is generated by $f_i$. Denote \[I_i=(\ddx{f_i}{x_i},\ddx{f_i}{y_i},\ddx{f_i}{z_i})=\{\CN[u,v]^{G_i},\CN[u,v]^{G_i}\},\] a Poisson commutator ideal of $\CN[u,v]^{G_i}$.  Choose a lift \[\psi\colon \CN[x_2,y_2,z_2]\to \CN[x_1,y_1,z_1]\] of inclusion $\CN[u,v]^{G_2}\subset \CN[u,v]^{G_1}$. Define $q=\pi_1(\frac{\psi(f_2)}{f_1})\in \CN[u,v]^{G_1}$. Denote by $\phi_{G_2,G_1}$ the map from $\CN[u,v]^{G_1}/I_1$ to $\CN[u,v]^{G_1}/\CN[u,v]^{G_1}I_2$ given by multiplication by $q$. Then $\phi_{G_1,G_2}$ is well-defined, does not depend on the choice of $\psi$ and is injective.
\end{prop}

Injectivity of $\phi_{G_1,G_2}$ indeed gives the desired contradiction. The fact that $\phi_{G_1,G_2}$ is well-defined and does not depend on the choice of $\psi$ is a direct computation. We will prove injectivity of $\phi_{G_1,G_2}$ in the next section.

\section{Injectivity of multiplication by $q$}
\label{InjSec}
I am grateful to Pavel Etingof for his help with rewriting this section.

Rename our subgroups: $H\subset G$ are finite subgroups of $\SL(2,\CN)$. Denote $\CN[u,v]$ by $A$. We will use Statement~\ref{StatSocle} in our proof: the socle of $A^H/\{A^H,A^H\}$ is one-dimensional and generated by an element of degree $d_H=\deg f_H-4$, where $f_H$ is a generator of the kernel of projection $\CN[x,y,z]\to A^H$. By definition $\deg q=\deg f_G-\deg f_H$, so $\deg q=d_G-d_H$.

We also note that $A^H/A^H\{A^G,A^G\}=A^H\otimes_{A^G} A^G/\{A^G,A^G\}$ and the map from $A^G/\{A^G,A^G\}$ to $A^H/A^H\{A^G,A^G\}$ is an embedding.

\begin{lem}
\label{LemCompositionPhi}
If $H\subset K\subset G\subset \SL(2,\CN)$ are finite subgroups then $\phi_{G,H}=(A^H\otimes_{A^K} \phi_{G,K})\circ \phi_{K,H}$.
\end{lem}
\begin{proof}
In the definition of $q_{G,H}$ take $\psi_{G,H}$ equal to $\psi_{G,K}\circ \psi_{K,H}$ and get $q_{G,H}=q_{G,K}q_{K,H}$. The lemma follows.
\end{proof}

We say that $(H,G)$ is {\it good} if $\phi_{G,H}$ defines an isomorphism between the socle of $A^H/\{A^H,A^H\}$ and the socle of $A^G/\{A^G,A^G\}\subset A^H/A^H \{A^G,A^G\}$. In order to prove Proposition~\ref{PropForMainLemma} it is enough to prove that all pairs $(H,G)$ are good.

\begin{prop}
Let $H\subset K\subset G$ be finite subgroups of $\SL(2,\CN)$. Suppose that $(H,K)$ is good. Then $(H,G)$ is good if and only if $(K,G)$ is good.
\end{prop}
\begin{proof}
We will use Lemma~\ref{LemCompositionPhi}. If $(K,G)$ is good then $(H,G)$ is good. Suppose that $(H,G)$ is good. Denote by $S_H$ the socle of $A^H/\{A^H,A^H\}$, similarly for $S_G,S_K$. We have $\phi_{K,H}(S_H)=S_K$, $\phi_{G,H}(S_H)=S_G$. Therefore $A^H\otimes_{A^K}\phi_{G,K}(1\otimes S_K)=S_G$, so $\phi_{G,K}(S_K)=S_G$ as desired.
\end{proof}

Let $C_2$ be a subgroup of $\SL(2,\CN)$ generated by a matrix $-1$.
\begin{prop}
\label{PropGoodForEven}
$(C_2,G)$ is good for any $G\supset C_2$.
\end{prop}
\begin{proof}
Let $A^G=\CN[X,Y,Z]/(F)$, $A^{C^2}=\CN[x,y,z]/(x^2-yz)$, $\pi$ be the projection from $\CN[x,y,z]$ to $A^{C_2}$. We choose a lift $\psi$ of embedding $A^G\subset A^{C_2}$ so that $\psi(X)=x P_X(y,z)+ Q_X(y,z)$, where $P_X,Q_X$ are polynomials. We have a similar equation for $\psi(Y),\psi(Z)$. 

Let $q_1=\frac{\psi(F)}{x^2-yz}$, $q=\pi(q_1)$. We see that the degree of $q$ equal to the degree of $F$ minus $4$. 

Since $\psi(F)=q_1(x^2-yz)$ we get $\psi(F)'_x=2xq_1+(q_1)'_x(x^2-yz)$, so $\pi(\psi(F)'_x)=2xq$.

Suppose that \[q=aF'_X+bF'_Y+cF'_Z.\] Multiplying by $2x$ we get \[F'_x=2axF'_X+2bxF'_Y+2cxF'_Z.\] Therefore $rF'_X+sF'_Y+tF'_Z=0$, where $r=X'_x-2ax$, similarly for $s,t$. We see that $\deg r=\deg X-2$, $\deg s=\deg Y-2$, $\deg t=\deg Z-2$.

From $rF'_X+sF'_Y+tF'_Z=0$ we get a derivation $D$ from $A^G$ to $A^{C_2}$ of negative degree given by $D(X)=r$, $D(Y)=s$, $D(Z)=t$. Corollary~\ref{NoDerivationOfHegativeDegCor} says that there are no derivations from $A^G$ to $A^{C_2}$ of negative degree. Therefore $r=s=t=0$.

We have $X'_x=P_X(y,z)$, so from $r=0$ we get $P_X(y,z)=2ax$. Therefore $P_X(y,z)$ is divisible by $yz$. Hence $\pi(\psi(X))$ belongs to the set $\CN+(u^2,v^2)\subset A$. We similarly deduce that $\pi(\psi(Y)),\pi(\psi(Z))$ belong to the same set. We deduce that $A^G\subset \CN+(u^2,v^2)\subset A$. 

Define a derivation $D$ from $A$ to $\CN(u,v)$ by $D(u)=\frac{1}{u}$, $D(v)=\frac{1}{v}$. We see that $D(\CN+(u^2,v^2))\subset A$, so the restriction of $D$ to $A^G$ is a derivation from $A^G$ to $A$. Theorem~\ref{MainDerThr} says that $D$ can be lifted to a unique derivation $D_1$ of $A$. Using the uniqueness part of the theorem for the derivation $uvD|_{A^G}$ we get that $uvD=uvD_1$, hence $D=D_1$, a contradiction: $D$ is not a derivation of $A$.
\end{proof}
\begin{prop}
\label{PropGoodForCyclic}
$(C_k,C_l)$ is good for any $k\mid l$.
\end{prop}
\begin{proof}
Let $l=km$.

We have $C_k=\CN[x,y,z]/(x^k-yz)$, $C_l=\CN[x,y,z]/(x^l-yz)$. Choose the following lift of $A^{C^l}\subset A^{C_k}$: \[\psi(x)=x,\qquad \psi(y)=y^m,\qquad \psi(z)=z^m.\] Since $\psi(x^l-yz)=x^l-y^mz^m$ we get $q=mx^{(m-1)k}$. The socle of $A^{C_k}/\{A^{C_k},A^{C_k}\}$ is generated by $x^{k-2}$. We have $qx^{k-2}=mx^{l-2}$. The proposition follows.
\end{proof}
\begin{prop}
All pairs $(H,G)$ are good.
\end{prop}
\begin{proof}
If both $G$ and $H$ have even order they contain $C_2$. In this case proposition follows from Lemma~\ref{LemCompositionPhi} and Proposition~\ref{PropGoodForEven}.

If both $G$ and $H$ have odd order then they are both cycic and $(H,G)$ is good by Proposition~\ref{PropGoodForCyclic}.

If $H=C_l$ has odd order and $G$ has even order then we have $H\subset K\subset G$, where $K$ is generated by $H$ and $C_2$ and is isomorphic to $C_{2l}$. The pair $(H,K)$ is good by~\ref{PropGoodForCyclic}. Both $K$ and $G$ have even order. We already proved that in this case $(K,G)$ is good. Hence $(H,G)$ is good by Lemma~\ref{LemCompositionPhi}.
\end{proof}
\section{Description of a universal commutative deformation.}
\label{NormalCaseSec}
Suppose that $G_1\triangleleft G_2$ are finite subgroups of $\SL(2,\CN)$. We are going to find a universal commutative deformation of $i\colon\CN[u,v]^{G_2}\subset\CN[u,v]^{G_1}$. This will be done in two steps:
\begin{enumerate}
\item
There exists a natural one-to-one correspondence between deformations of $i$ and deformations of $\CN[u,v]^{G_1}$ that admit an action of $G_2/G_1$ with certain properties.
\item
There exists a universal object among deformations of $\CN[u,v]^{G_1}$ that admit an action of $G_2/G_1$.
\end{enumerate}

Suppose that $A$ is a graded algebra, $G$ is a group of automorphisms of $A$. Then $G$ acts on isomorphsm classes of deformations of $A$: if $g$ is an element of $G$, $(\mc{A},\chi\colon \mc{A}/\mc{A}\mm\cong A)$ is a deformation of $A$ over $B$, we define $^g\mc{A}$ as $(\mc{A},g\circ\chi)$.

Suppose that $i\colon A_2\to A_1$ is an inclusion of graded algebras, $G$ is a group of automorphisms of $A_1$ that preserve $A_2$ element-wise. If $F\colon \mc{A}_2\to\mc{A}_1$ is a deformation of $i$, then the same map between $\mc{A}_1$ and $^g\mc{A}_2$ will be deformation of $i$. Therefore we have an action of $G$ on isomorphsm classes of deformations of $i$.

Denote $\CN[u,v]$ by $A$. Denote by $i$ the inclusion $A^{G_2}\subset A^{G_1}$.

Suppose that $F\colon\mc{A}_2\to\mc{A}_1$ is a deformation of $i$, $g$ is an element of $G=G_2/G_1$. Then $F\colon\mc{A}_2\to^g\mc{A}_1$ is a deformation of $i$. Applying Theorem~\ref{MainTheorem} to these two deformations we get the following proposition:
\begin{prop}
Suppose that $G_1$ is a normal subgroup of $G_2$ and $F\colon \mc{A}_2\to\mc{A}_1$ is a deformation of $i\colon A^{G_2}\to A^{G_1}$ over $R$. Then for every $g\in G$ there exists a unique $R$-linear isomorphism of deformations $\tau_g\colon ^g\mc{A}_1\to\mc{A}_1$ such that $\tau_g F=F$.
\end{prop}
\begin{cor}
\label{DeformationToActionCor}
There exists an $R$-linear action of $G$ on $\mc{A}_1$ such that
\begin{enumerate}
\item
$G$ acts on the image of $\mc{A}_2$ trivially.
\item
The isomorphism $\chi\colon\mc{A}_1/\mc{A}_1\mm\to \CN[u,v]^{G_1}$ intertwines the action of $G$.
\end{enumerate}
\end{cor}
\begin{proof}
Suppose that $g$ is an element of $G$. Then we have an isomorphism of deformations $\tau_g\colon ^g\mc{A}_1\to \mc{A}_1$ such that $\tau_g F=F$. Since graded algebras $^g\mc{A}_1$ and $\mc{A}_1$ are equal as sets, we have an isomorphism of graded algebras $\rho_g\colon\mc{A}_1\to\mc{A}_1$ such that $\rho_g|{\mc{A}_2}=\id$. Suppose that $h$ is an element of $G$. We see that $\tau_g$, considered as a map from $^g{^h}\mc{A}_1$ to $^h\mc{A}_1$ is an isomorphism of deformations. Hence $\tau_g\circ\tau_h\colon ^{gh}\mc{A}_1\to \mc{A}_1$ is an isomorphism of deformations, so $\tau_g\circ\tau_h=\tau_{gh}$. It follows that $\rho$ is an action of $G_2/G_1$ on $\mc{A}_1$.

Denote by $\ovl{\rho}$ the corresponding action of $G_2/G_1$ on $\mc{A}_1/\mm\mc{A}_1$. Denote by $p$ the projection $\mc{A}_1\to\mc{A}_1/\mm\mc{A}_1$.

Since $\tau_g$ is an isomorphism of deformation, $\chi\circ p\circ\tau_g=g\circ \chi\circ p$. On the other hand $\chi\circ p\circ\tau_g=\chi\circ p\circ\rho_g=\chi\circ\ovl{\rho_g}\circ p$. Hence $\chi\ovl{\rho_g}=g \chi$. Hence $\rho$ is an action of $G$ on $\mc{A}_1$ that satisfies both properties.
\end{proof}

Note that we can go in another direction, from the certain action of $G$ to a deformation of $i\colon A^{G_2}\to A^{G_1}$:
\begin{prop}
\label{PropActionToDeformation}
Suppose that $\mc{A}_1$ is a (possibly, noncommutative) deformation of $A^{G_1}$ and there exists an $R$-linear action of $G$ on $\mc{A}_1$ such that the isomorphism $\chi\colon\mc{A}_1/\mc{A}_1\mm\to A^{G_1}$ is an intertwining operator.  Then $\mc{A}_1^{G_2/G_1}$ is a deformation of $A^{G_2}$ over $R$ and the inclusion $F\colon \mc{A}_1^{G_2/G_1}\to \mc{A}_1$ is a deformation of $i$.
\end{prop}

\begin{defn}
If such an action exists we say that $\mc{A}_1$ admits a good action of $G$.
\end{defn}
It follows from the proof of Corollary~\ref{DeformationToActionCor} that a collection of $R$-linear isomorphisms of deformations $\tau_g\colon ^g\mc{A}_1\to \mc{A}_1$ gives a good action of $G$ on $\mc{A}_1$.

Let $F\colon \mc{A}_2\to\mc{A}_1$ be a deformation of $i$. From Corollary~\ref{DeformationToActionCor} we get a good action of $G$ on $\mc{A}_1$. From Proposition~\ref{PropActionToDeformation} we see that $\mc{A}_1^G$ is a deformation of $A^{G_2}$. Since $G$ acts trivially on $F(\mc{A}_2)$, the image of $F$ is contained in $\mc{A}_1^G$. The map $F\colon \mc{A}_2\to\mc{A}_1^G$ is an $R$-linear morphism of deformations of $A^{G_2}$. From Lemma~\ref{DeformationOfIdLem} we get that $F$ is an isomorphism of deformations of $A^{G_2}$.

Hence we proved the following proposition:
\begin{prop}
\label{PropAnyDeformationIsInclusionOfGInvariants}
Let $F\colon \mc{A}_2\to\mc{A}_1$ be a deformation of $i\colon A^{G_2}\to A^{G_1}$. Then $\mc{A}_1$ admits a good action of $G$ and $F$ is isomorphic to $\mc{A}_1^{G_2}\subset \mc{A}_1$.
\end{prop}

If we forget about $\mc{A}_2$ the set of morphisms does not change:
\begin{prop}
\label{ForgettingSecondArrowIsInjectiveProp}
Let $F_1\colon \mc{A}_2\to \mc{A}_1$, $F_2\colon \mc{B}_2\to \mc{B}_1$ be deformations of $i$. Suppose that $\phi$ is a morphism of deformations of $A^{G_1}$ from $\mc{A}_1$ to $\mc{B}_1$. Then there exists a unique morphism $\psi\colon \mc{A}_2\to\mc{B}_2$ of deformations of $A^{G_2}$ such that $(\psi,\phi)$ is a morphism of deformations of $i$.
\end{prop}
\begin{proof}
We can assume that $\mc{A}_2=\mc{A}_1^G$, $\mc{B}_2=\mc{B}_1^G$.

We see that $\psi$, if it exists, must be equal to $\phi|_{\mc{A}_2^1}$, so it is enough to prove that $\phi$ intertwines the action of $G$. This is proved in the next lemma.
\end{proof}
\begin{lem}
\label{LemAnyMorphismIntertwinesGoodAction}
Suppose that $\mc{A},\mc{B}$ are deformations of $A^{G_1}$ over $R,S$ with a good action of $G$. Then any morphism of deformations $\phi\colon\mc{A}\to\mc{B}$ intertwines the action of $G$.
\end{lem}
\begin{proof}
Let $g$ be an element of $G$. Denote by $\tau_g$ the isomorphism between $^g\mc{A}$ and $\mc{A}$, by $\psi_g$ the isomorphism between $^g\mc{B}$ and $\mc{B}$. The map $\phi$ is a morphism of deformations from $^g\mc{A}\to ^g\mc{B}$. 

Since $\psi_g$  is $R$-linear and $\tau_g$ is $S$-linear we get that two maps $\psi_g\phi$ and $\phi\tau_g$ are morphisms of deformations from $^g\mc{A}$ to $\mc{B}$ and their restrictions on $R$ are equal. Using Corollary~\ref{EqualIffRestictionsAreEqualCor} we get that $\psi_g\phi=\phi\tau_g$. The lemma follows.
\end{proof}

Let $\mc{A}_0$ be a universal deformation of $A^{G_1}$ over the base $R_0$. Suppose that $g$ is an element of $G$. Then $^g\mc{A}_0$ is a universal deformation too. Hence there is a unique isomorphism $\tau_g\colon ^g\mc{A}_0\to \mc{A}_0$. Restricting $\tau_g$ to $R_0$ we get an action of $G$ on $R_0$. 

Recall that $\mc{A}_0=\CN[x,y,z,t_1,\ldots,t_m]/(f-\sum t_i u_i)$, where $\CN[u,v]^{G_1}=\CN[x,y,z]/(f)$ and $u_1,\ldots,u_m$ is a basis of $\CN[x,y,z]/(f'_x,f'_y,f'_z)$.

It follows that $R_0=\CN[t_1,\ldots,t_m]$ is a polynomial algebra, so we can write $R_0=\CN[V]$ for some vector space $V$. Hence $G$ acts on $V$. There is a unique decomposition of $V^*$ into subrepresentations \[V^*=(V^*)^G\oplus (V^*)_G\] where $(V^*)^G$ is the subspace of $G$-invariants. Let $I$ be an ideal in $R_0$ generated by $(V^*)_G$. We note that $e(V^*)_G=\{0\}$ where $e=\frac{1}{\lvert G\rvert}\sum_{g\in G}g$ is an idempotent in $\CN[G]$. 

We formulate a lemma for future use.

\begin{lem}
\label{LemTakingG2G1Invariants}
Let $R_0=\CN[V]$ be a base of a universal deformation, $I$ be an ideal generated by $(V^*)_G\subset R_0$. Then $R_0/IR_0\cong \CN[V^G]$.
\end{lem}

Now we are ready to describe a universal commutative deformation of $i\colon A^{G_2}\to A^{G_1}$.

\begin{prop}
\label{UniversalIsIntertwiningProp}
\begin{enumerate}
\item
Suppose that $\mc{A}$ is a deformation of $A^{G_1}$ over $R$ with a good action of $G$. Then the morphism of deformations $\psi$ from $\mc{A}_0$ to $\mc{A}$ is $G$-equivariant. Moreover, $\psi$ factors through $\mc{A}_0/I\mc{A}_0$.
\item
$\mc{B}_0=\mc{A}_0/I\mc{A}_0$ admits a good action of $G$.
\item
$\mc{B}_0^G\subset \mc{B}_0$ is a universal deformation of $i$.
\end{enumerate} 
\end{prop}
\begin{proof}
\begin{enumerate}
\item
Suppose that $g$ is an element of $G$. We see that $\psi\colon ^g\mc{A}_0\to ^g\mc{A}$ is a morphism of deformations. The morphism of deformations from $^g\mc{A}_0$ to $\mc{A}$ is unique, so $\psi\tau_g=\tau_g\psi$. It follows that $\psi$ is $G$-equivariant. By definition the action of $G$ fixes $R$. Using that $\psi(V^*)\subset R$ we get \[\psi((V^*)_G)=e\psi((V^*)_G)=\psi(e(V^*)_G)=\{0\}.\] It follows that that $\psi(I)=0$, so $\psi$ factors through $\mc{A}_0/I\mc{A}_0$.
\item
By definition $G(I)=I$, so $G$ acts on $\mc{B}_0$. The isomorphism $\chi\colon \mc{A}_0/(R_0)_{>0}\mc{A}_0\cong A^{G_1}$ intertwines the action of $G$. Since $I$ is contained in $(R_0)_{>0}$ the same holds for $\mc{B}_0$ instead of $\mc{A}_0$. It follows from Lemma~\ref{LemTakingG2G1Invariants} that $G$ acts trivially on $R_0/IR_0=\CN[V^G]$. Hence the action of $G$ on $\mc{B}_0$ is good.
\item
Consider category of deformations of $A^{G_1}$ with a good action of $G$. We see that $\mc{B}_0$ is an initial object in this category. The statement follows from Proposition~\ref{ForgettingSecondArrowIsInjectiveProp}
\end{enumerate}
\end{proof}

\begin{rem}
Using Theorem~\ref{MainTheorem} we can find a universal commutative deformation of $A^{G_2}\subset A^{G_1}$ in the case when there exists a chain of normal inclusions $G_1=H_1\triangleleft H_2\triangleleft\cdots\triangleleft H_k=G_2$. However, this requires some computation, so we omit it now.
\end{rem}
\subsection{Examples.}
We include two examples of universal commutative deformations.

Let us describe universal deformations of inclusions $A^{C_n}\subset A^{C_{nk}}$ and $A^{C_{2n}}\subset A^{\Dic_n}$. Here $C_m$ is generated by $\begin{pmatrix}
e^{\frac{2\pi i}{m}} & 0\\
0 & e^{-\frac{2\pi i}{m}}
\end{pmatrix}$ and $\Dic_n$ is generated by $C_{2n}$ and $\begin{pmatrix}
0 & 1\\
-1 & 0
\end{pmatrix}$. 

The general algorithm of describing universal deformation of $A^{G_2}\subset A^{G_1}$ is as follows:
\begin{enumerate}
\item
Describe the action of $G$ on $A^{G_1}$.
\item
Let $(\mc{A}_0,\chi)$ be a universal deformation of $A^{G_1}$ over $R_0$. We lift the action of $G$ to $\mc{A}_0$ so that $\chi$ becomes an intertwining operator.
\item
We get an action of $G$ on $R_0=\CN[V]$. Let $I$ be the kernel of the map $\CN[V]\to \CN[V^G]$.
\item
Let $\mc{B}_0=\mc{A}_0/I\mc{A}_0$. From the action of $G$ on $\mc{A}_0$ we get a good action of $G$ on $\mc{B}_0$ and $\mc{B}_0^G\subset\mc{B}_0$ is a universal commutative deformation of $A^{G_2}\subset A^{G_1}$.
\end{enumerate}
\paragraph{$C_{nk}\subset C_n$.}
Let $(\mc{A}_0,\chi)$ be a universal deformation of $\CN[u,v]^{C_n}$ over $R_0$. Then $R_0=\CN[a_0,\ldots,a_{n-2}]$, $\mc{A}_0=R_0[x,y,z]/(x^n+\suml_{i=0}^{n-2}a_ix^i-yz)$, $\chi\colon \mc{A}_0/B^{>0}\mc{A}_0\cong \CN[u,v]^{C_n}=\CN[x,y,z]/(x^n-yz)$ sends $x,y,z$ to $x,y,z$.

\begin{lem}
The action of $G=C_k$ on $\mc{A}_0$ is $R_0$-linear. Inclusion $\mc{A}_0^{C_k}\subset\mc{A}_0$ is a universal deformation of $\CN[u,v]^{C_{nk}}\subset\CN[u,v]^{C_n}$.
\end{lem}
\begin{rem}
We can write $\mc{A}_0^{C_k}$ explicitly: $\mc{A}_0^{C_k}\cong \CN[x,y,z,a_0,\ldots,a_{n-2}]/((x^n+\suml_{i=0}^{n-2} a_ix^i)^k-yz)$. Inclusion $\mc{A}_0^{C_k}\subset\mc{A}_0$ is given by $x\mapsto x$, $y\mapsto y^k$, $z\mapsto z^k$.
\end{rem}
\begin{rem}
We see that any deformation of $\CN[u,v]^{C_n}$ appears in some deformation of $\CN[u,v]^{C_{nk}}\subset \CN[u,v]^{C_n}$.
\end{rem}
\begin{proof}
Let $g$ be a generator of $G$ that is equal to an image of $\begin{pmatrix}
\eps & 0\\
0 & \eps^{-1}
\end{pmatrix}\in C_{nk}$ in $G$. Here $\eps=e^{\frac{2\pi i}{nk}}$. The action of $G$ on $A^{C_n}$ is obtained from the action of $C_{nk}$ on $A$, so $gx=g(uv)=uv=x$, $gy=g(u^n)=\eps^n y$, $gz=g(v^n)=\eps^{-n} z$. 

Let $G$ act on $\mc{A}_0$ as follows: $G$ fixes $R_0$, $gx=x$, $gy=\eps^n y$, $gz=\eps^{-n}z$. This is a well-defined action and $\chi$ is an intertwining operator.

Since $G$ acts on $R_0$ trivially we have $I=\{0\}$ and $\mc{B}_0=\mc{A}_0$. Hence $\mc{A}_0^{C_k}\subset\mc{A}_0$ is a universal deformation of $\CN[u,v]^{C_{nk}}\subset\CN[u,v]^{C_n}$.
\end{proof}
\paragraph{$C_{2n}\subset\mathbb{D}_n$.}
The universal deformation $(\mc{A}_0,\chi)$ of $\CN[u,v]^{C_{2n}}$ is given by $\mc{A}_0= \CN[x,y,z,a_0,\ldots,a_{2n-2}]/(x^{2n}+\suml_{i=0}^{2n-2} a_ix^i-yz)$.

\begin{lem}
The nontrivial element of $\mathbb{D}_n/C_{2n}=C_2$ acts on $R_0$ as follows: $a_i\mapsto (-1)^i a_i$. Hence
\begin{enumerate}
\item
$I=(a_1,a_3,\ldots,a_{2n-3})$.
\item
\(R_0/IR_0\cong\CN[a_0,a_2,\ldots,a_{2n-2}]\)
\item
\(\mc{A}_0/I\mc{A}_0\cong\CN[a_0,\ldots,a_{2n-2},x,y,z]/(x^{2n}+\suml_{i=0}^{n-1}a_{2i}x^{2i}-yz)\)
\end{enumerate}

The universal commutative deformation of $\CN[u,v]^{\mathbb{D}_n}\subset\CN[u,v]^{C_{2n}}$ is given by \begin{multline*}\CN[X,Y,Z,a_0,a_2,\ldots,a_{2n-2}]/(XY^2-Z^2-4X^{n+1}-\suml_{i=0}^{n-1}a_{2i}X^{i+1})\to \\
\CN[x,y,z,a_2,\ldots,a_{2n-2}]/(x^{2n}+\suml_{i=0}^{n-1}a_{2i}x^{2i}-yz)
\end{multline*} where $X\mapsto x^2$, $Y\mapsto y+z$, $Z\mapsto x(y-z)$.
\end{lem}
\begin{proof}
The generator $g$ of $G=C_2$ acts on $A^{C_{2n}}=\CN[x,y,z]/(x^{2n}-yz)$ as matrix $\begin{pmatrix}
0 & 1\\
-1 & 0
\end{pmatrix}$, so $gx=g(uv)=-vu=-x$, $gy=g(u^{2n})=z$, $gz=g(v^{2n})=y$.

We can lift this action to $\mc{A}_0$ as follows: $gx=-x$, $gy=z$, $gz=y$, $ga_i=(-1)^i a_i$. It follows that $I=(a_1,\ldots,a_{2n-3})$.

We see that $\mc{B}_0$ is generated by $X=x^2$, $Y=y+z$, $Z=x(y-z)$ over $\CN[a_0,\ldots,a_{2n-2}]$. They satisfy 
\begin{multline*}
XY^2-Z^2-4X^{n+1}=x^2(y+z)^2-x^2(y-z)^2-4x^{2n+2}=\\
4x^2(yz-x^{2n})=4x^2(\suml_{i=0}^{n-1}a_{2i}x^{2i})=\suml_{i=0}^{n-1}a_{2i}X^{i+1}
\end{multline*}
Since $\mc{B}_0^{C_2}$ is a deformation of $A^{\Dic_n}$ this is the only relationship between $X,Y,Z$. Hence $\mc{B}_0^{C_2}$ is isomorphic to $\CN[X,Y,Z,a_0,a_2,\ldots,a_{2n-2}]/(XY^2-Z^2-4X^{n+1}-\suml_{i=0}^{n-1}a_{2i}X^{i+1})$. The lemma follows.
\end{proof}
\begin{rem}
The universal commutative deformation of $\CN[x,y]^{\mathbb{D}_n}$ is given by \[\CN[x,y,z,a_0,\ldots,a_n,b]/(xy^2-z^2-4x^{n+1}-\suml_{i=0}^n a_ix^i-by).\] We see that there exist deformations of $\CN[x,y]^{\mathbb{D}_n}$ and $\CN[x,y]^{C_{2n}}$ that do not appear in deformations of $\CN[x,y]^{\mathbb{D}_n}\subset \CN[x,y]^{C_{2n}}$.
\end{rem}

\section{CBH algebras}
\label{CBHSec}
From now on deformations are not supposed to be commutative.

\subsection{Plan of Sections~\ref{CBHSec}-\ref{NoncommutativeNormalSec}}
Let us write a short plan of Sections~\ref{CBHSec}-\ref{NoncommutativeNormalSec}. First, we define a notion of a Crawley--Boevey---Holland algebra and recall basic properties of CBH algebras. We introduce CBH algebras because they provide a reasonable way to parametrize noncommutative deformations of the algebras $\CN[u,v]^G$. The classification of noncommutative deformations in~\cite{Loseu} is formulated in terms of CBH parameters. 

In the case of normal inclusion $G_1\triangleleft G_2$ there exist certain inclusions of CBH algebras that deform the inclusion $i\colon\CN[u,v]^{G_2}\subset \CN[u,v]^{G_1}$. Commutative deformations of $\CN[u,v]^G$ are parametrized by $V/W$, where $V$ is a space with a root system, $W$ is the corresponding Weyl group. In the end of Section~\ref{CBHSec} we prove a similar result about deformations of $i$: commutative deformations of $i$ are parametrized by $V/W$, where $V$ is a space with a root system, $W$ is a corresponding Weyl group. 

In Section~\ref{NoncommutativeNormalSec} we introduce a noncommutative deformation of $i$ over $\CN[V/W]\otimes \CN[z]$ that is isomorphic to universal commutative deformation of $i$ when we set $z=0$. Then we prove that this is a universal deformation of $i$.
\subsection{Definition and basic properties of CBH algebras}
\begin{defn}
\label{SmashProductDef}
Suppose that $G$ is a finite group acting on an algebra $A$ by automorphisms. Define a bilinear product $\cdot$ on $A\otimes_{\CN} \CN[G]$ in the following way: $(a\otimes g)\cdot (b\otimes h)=ag(b)\otimes gh$. This algebra is called the smash product of $A$ and $G$ and is denoted by $A\#G$.
\end{defn}
We see that $\cdot$ is an associative product.

\begin{defn}
Let $R$ be a graded $\CN$-algebra, $G$ be a finite subgroup of $\SL(2,\CN)$. We have a grading on $R[G]$ such that elements of $G$ are homogeneous of degree $0$. Suppose that $c$ is an element of $Z(R[G])$ of degree $2$, $e$ is the element of $R[G]$ equal to $\frac{1}{|G|}\suml_{g\in G} g$. The algebra $e(R\langle x,y \rangle\#G/(xy-yx-c))e$ is called a CBH algebra with parameter $c$ and is denoted by $\mathcal{O}^R_c$ or simply $\mathcal{O}_c$.
\end{defn}
The algebra $\mathcal{O}_c$ is a graded unital algebra.

The following facts were proved in~\cite{CBH}
\begin{stat*}
\begin{enumerate}
\item
$\mathcal{O}_c$ is a free $R$-module.
\item
Suppose that $c=\sum_{g\in G}c_g g$. Then $\mathcal{O}_c$ is commutative if and only if $c_1=0$.
\end{enumerate}
\end{stat*}
Suppose that $R_0\cong\CN$. It follows that $\mathcal{O}_c$ is a flat deformation of $e(\CN[u,v]\#G)e$.
\begin{rem}
\label{SphericalSubalgebraRem}
The map $a\mapsto ea$ is an isomorphism of unital algebras between $\CN[u,v]^G$ and $e(\CN[u,v]\#G)e$.
\end{rem}

Hence $\mathcal{O}_c$ is a deformation of $\CN[u,v]^G$ over $R$.

Recall that $Z(\CN[G])^*$ has a basis consisting of characters of irreducible $\CN[G]$-modules. Denote them by $\chi_0,\chi_1,\ldots,\chi_n$, where $\chi_0$ is the character of the trivial representation. Denote by $\chi_{\CN^2}$ the character of the tautological representation of $G$ on $\CN^2$. Denote by $(\cdot,\cdot)$ the standard scalar product on $Z(\CN[G])^*$. We have another Hermitian form: $B(\chi_i,\chi_j)=(\chi_i,\chi_{\CN^2}\otimes\chi_j)$.

The following theorem is well-known.
\begin{thr*}[McKay]
$(B(\chi_i,\chi_j))_{i,j=1\ldots m}$ is a Cartan matrix of some simply laced Dynkin diagram. The form $B$ is positive semidefinite, its kernel is generated by the character of regular representation.
\end{thr*}

Hence $\chi_0,\ldots,\chi_m$ form an affine roots system with respect to $B$ and $\chi_1,\ldots,\chi_m$ form a root system with respect to $B$. Denote the corresponding finite Weyl group by $W$, it acts on $Z(\CN[G])^*$. We have a dual affine root system in $Z(\CN[G])$ and a dual action of $W$ on $Z(\CN[G])$. For every commutative graded algebra $R$, $W$ acts on $Z(R[G])$ respecting grading. 

We will need another fact from~\cite{CBH}.
\begin{thr}
$\mc{O}^R_c$ is naturally isomorphic to $\mc{O}^R_{wc}$ for all commutative graded algebras $R$, $c\in Z(R[G])$ of degree $2$, $w\in W$.
\end{thr}
\begin{cor}
\label{WeylActuallyActsCor}
Suppose that $R$ is a graded algebra, $c$ is an element of $Z(R[G])$ of degree $2$, $H$ is a subgroup of $W$ that acts on $R$ via $h\mapsto \phi_h$. If $\phi_h(c)=h^{-1}(c)$ for all $h\in H$ then $\phi$ can be lifted to an action of $H$ on $\mc{O}^R_c$ by automorphisms of deformations.
\end{cor}
\begin{proof}
Let $R_h$ be the following $R$-module: $R$ acts on itself by $r.s=\phi_{h^{-1}}(r)s$. Base change $\mc{O}_c\to \mc{O}_c\otimes_R R_h$ is a morphism of deformations by Statement~\ref{BaseTensoringStat}. We have \[\mc{O}_c\otimes_R R_h=\mc{O}_{\phi_h(c)}\cong \mc{O}_c.\]  This gives an action of $H$ on $\mc{O}_c$  
\end{proof}

\subsection{Connection between CBH algebras and universal commutative deformation}
Let $n+1$ be the number of conjugacy classes in $G$. Suppose that $C_0=\{1_G\},C_1,\ldots,C_n$ are all conjugacy classes in $G$. Then $1_G$, $g_1=\sum_{g\in C_1}g$, $g_2$, \dots, $g_n=\sum_{g\in C_n}g$ is a basis of $Z(\CN[G])$. Consider the CBH algebra $\tilde{\mathcal{O}}$ with parameter $\sum_{i=1}^n z_ig_i\in Z(\CN[z_1,\ldots,z_n][G])$ , where each $z_i$ has degree $2$. Note that $\tilde{\mc{O}}$ is commutative. Using Corollary~\ref{WeylActuallyActsCor} we see that $W$ acts on $\tilde{\mc{O}}$ by automorphisms of deformation.

Suppose that $\mc{A}_0$ is a universal commutative deformation of $\CN[u,v]^G$. Let $\chi$ be a unique morphism of deformations from $\mc{A}_0$ to $\tilde{\mathcal{O}}$. 
\begin{thr}[Crawley--Boevey---Holland, Kronheimer]
\label{UniversalCBHAlgebraThr}
\begin{enumerate}
$\chi$ is a bijection between $\mc{A}_0$ and $\tilde{\mathcal{O}}^W$.
\end{enumerate}
\end{thr}
\begin{proof}
The fact that $\operatorname{Specm}\tilde{\mc{O}}^W\twoheadrightarrow \Specm(\CN[z_0,\ldots,z_m]^W)$ is a universal deformation of $\Specm\CN[u,v]^G$ in the category of complex analytic varieties was proved in~\cite{CBH} and~\cite{Kronheimer}, see discussion at the end of Section 8 of~\cite{CBH}. It follows that there exists a complex-analytic morphism of deformations $\phi$ from $\Specm\mc{A}_0$ to $\operatorname{Specm}\tilde{\mc{O}}^W$. 

Since $\CN[u,v]^G$ is a graded algebra, we have an action of $\CN^{\times}$ on $\CN[u,v]^G$. So we have an algebraic/complex-analytic action of $\CN^{\times}$ on a universal algebraic/complex-analytic deformation of $\Specm\CN[u,v]^G$. So we have an action of $\CN^{\times}$ on $\Specm\tilde{\mc{O}}^W$ and $\Specm\mc{A}_0$. It is not hard to prove that this action coincides with the action of $\CN^{\times}$ coming from grading on $\tilde{\mc{O}^W}$ and $\mc{A}_0$.

Since $\Specm\tilde{\mc{O}}^W$ is a universal deformation, $\phi$ intertwines the action of $\CN^{\times}$. Suppose that $f$ is a homogeneous element of $\tilde{\mc{O}}^W$ of degree $d$. This means that for any $x\in\Specm\tilde{\mc{O}}^W$, $z\in\CN^{\times}$, $f(zx)=z^d f(x)$. So $h=f\circ\phi$ is a complex-analytic function on $\Specm\mc{A}_0$ such that for any $s\in\Specm\mc{A}_0$, $z\in \CN^{\times}$, $h(zs)=z^d h(s)$.

Recall that $\mc{A}_0=\CN[x,y,z,a_1,\ldots,a_m]/(f(x,y,z)-\sum a_iu_i(x,y,z))$. So we can write $h(s)$ in some neighborhood of zero as convergent series in variables $x,y,z,a_1,\ldots,a_m$. We see that changing $s$ to $zs$ results in multiplicating the coefficient on $x^{\alpha_x}y^{\alpha_y}\ldots a_m^{\alpha_m}$ by $z^{\alpha_x+\alpha_y+\ldots+\alpha_m}$. It easily follows from $h(zs)=z^d h(s)$ that $h$ can be written using monomials with $\alpha_x+\ldots+\alpha_m=d$. In others words $h$ is a polynomial.

We see that $\phi$ is a morphism of algebraic varieties. Denote by $\chi^*$ the morphism of algebraic varieties corresponding to $\chi$. Since $\Specm\mc{A}_0$ and $\Specm\tilde{O}^W$ are universal deformations, both compositions $\phi\chi^*$ and $\chi^*\phi$ are identity. Hence $\chi$ is an isomorphism.
\end{proof}

Let $A=\CN[u,v]$, $G=G_2/G_1$.

Now we consider deformations of inclusion $A^{G_2}\subset A^{G_1}$. We want to prove a theorem similar to Theorem~\ref{UniversalCBHAlgebraThr}. First we will show that CBH algebras can be used to construct a deformation of $A^{G_2}\subset A^{G_1}$. 

\begin{prop}
\label{PropConstructionOfCBHInclusion}
Suppose that $R$ is a graded $\CN$-algebra, $G_1\triangleleft G_2$ are finite subgroups of $\SL(2,\CN)$, $c$ is an element of $Z(R[G_1])\cap Z(R[G_2])$ of degree $2$, $\mc{O}_c^1$ and $\mc{O}_c^2$ are CBH algebras for groups $G_1,G_2$ with parameter $c$. Then there exists an embedding of $\mc{O}_c^2$ into $\mc{O}_c^1$. This embedding is a deformation of $A^{G_2}\subset A^{G_1}$ over $R$.
\end{prop}
\begin{proof}
Define an action of $G_2$ on $R\langle u,v\rangle\#G_1$ as follows: $g(f\otimes h)=gf\otimes ghg^{-1}$. This is an action by $R$-algebra automorphisms. We see that $g(xy-yx-c)=xy-yx-c$ and $ge_{G_1}=e_{G_1}$. So we have an action of $G_2$ on $\mc{O}_c^1=e_{G_1}(B\langle u,v\rangle \# G_1/(uv-vu-c))e_{G_1}$. Algebra $\mc{O}_c^1$ consists of elements $f\otimes e_{G_1}$, where $f\in \CN[u,v]^{G_1}$, so the action of $G_1$ on $\mc{O}_c^1$ is trivial. Hence we have an action of $G$ on $\mc{O}^1_c$. We see that this action is good, so by Proposition~\ref{PropActionToDeformation} $(\mc{O}^1_c)^{G}\subset \mc{O}^1_c$ is a deformation of $A^{G_2}\subset A^{G_1}$.

Using Remark~\ref{SphericalSubalgebraRem} see that $(\mc{O}_c^1)^{G}\cong e_{G}(\mc{O}_c^1\# G)e_{G}$. Now it is easy to construct an isomorphism of deformations between \[e_{G}(\mc{O}_c^1\# G)e_{G}\] and \[\mc{O}_c^2=e_{G_2}(\CN\langle u,v\rangle\# G_2)e_{G_2}.\]
\end{proof}



Recall that $g_0=1_{G_1}$, $g_i=\sum_{g\in C_i} g$, where $C_i$ are all conjugacy classes in $G$.
\begin{lem}
\label{G2G1ActsByDiagramAutomLem}
\begin{enumerate}
\item
There exists a root system in $\Span(g_1,\ldots,g_m)$ such that the action of $G$ on $Z(\CN[G_1])$ by conjugation permutes simple roots and preserves scalar product.
\item
This action lifts to an action of $G$ on $\tilde{\mc{O}}$ such that the natural map $\tilde{\mc{O}}\to \CN[u,v]^{G_1}$ is an intertwining operator.
\end{enumerate}
\end{lem}
\begin{proof}
We have $\Span(g_1,\ldots,g_m)^{\perp}=\CN\chi_{reg}$, where $\chi_{reg}$ is the character of regular representation of $G$. Since $\chi_{reg}=\sum_{i=0}^m a_i\chi_i$ where all $a_i>0$ we deduce that the pairing between $\Span(g_1,\ldots,g_m)$ and $\Span(\chi_1,\ldots,\chi_m)$ is nondegenerate. 

Hence from the root system given by simple roots $\{\chi_1,\ldots,\chi_m\}$ and the action of $W$ we get a dual root system in $\Span(g_1,\ldots,g_m)$ and an action of $W$.

It is enough to prove that a dual action of $G$ on $Z(\CN[G_1])^*$ permutes simple roots and preserves scalar product. For a representation $\rho$ we have $g\chi_{\rho}=\chi_{\rho\circ g^{-1}}$, hence the action of $G$ permutes simple roots. The tautological action of $G_1$ on $\CN^2$ can be extended to an action of $G_2$, hence $g\chi_{\CN^2}=\chi_{\CN^2}$. Since the action of $G$ preserves the standard product $(\cdot,\cdot)$ it follows that the action of $G$ preserves scalar product $B(\chi_i,\chi_j)=(\chi_i\otimes \CN^2, \chi_j)$.

Consider the following action of $G_2$ on $\CN[z_1,\ldots,z_m]\langle u,v\rangle\#G_1$: $g.h=ghg^{-1}$ for $h\in G_1$. If $C_i$, $C_j$ are conjugacy classes in $G_1$ such that $gC_ig^{-1}=C_j$, then $gz_i=z_j$. The action of $G_2$ on $\Span(u,v)$ is tautological. We see that this action is well defined and $g(xy-yx)=xy-yx$, $gc=g(\sum_{i=1}^m z_m\sum_{h\in C_m}h)=c$, $ge_{G_1}=e_{G_1}$. Hence $G_2$ acts on $\tilde{O}$ and the action of $G_1\subset G_2$ is trivial. Therefore we get an action of $G$ on $\tilde{O}$. The map $\tilde{\mc{O}}\twoheadrightarrow \CN[u,v]^{G_1}$ intertwines the action of $G$ by construction.
\end{proof}


Let $Z(\CN[G])=V$, then the base of the deformation $\tilde{\mathcal{O}}$ is naturally isomorphic to $\CN[V]$. We deduce from theorem~\ref{UniversalCBHAlgebraThr} that $R_0$, the base of $\mc{A}_0$, is isomorphic to $\CN[V/W]$. 

In Section~\ref{NormalCaseSec} we introduced an action of $G$ on $R_0$, so $G$ acts on $V/W$. It follows from universality of $\mc{A}_0$ that the natural projection from $V$ to $V/W$ intertwines the action of $G$. 

Recall that there is a good action of $G$ on $\mc{B}_0=\mc{A}_0\otimes_{R_0}\CN[(V/W)^G]$ and $\mc{B}_0^G\subset \mc{B}_0$ is a universal deformation of $A^{G_2}\subset A^{G_1}$.


Suppose that $1_{G_1},S_1,\ldots,S_k$ are the orbits of $G_2$-action on $G_1$. Then $1_{G_1},h_1=\sum_{g\in S_1}g,\ldots,h_k=\sum_{g\in S_k}g$ is a basis of $Z(\CN[G_1])\cap Z(\CN[G_2])$. Consider the CBH algebra with parameter $\suml_{i=1}^k t_ih_i\in Z(\CN[t_1,\ldots,t_k][G])$, denote it by $\mc{B}_1$. 

There is a $\CN[t_1,\ldots,t_k]$-linear action of $G_2$ on $\CN[t_1,\ldots,t_k]\langle x,y\rangle\# G_1$: $G_2$ acts on $x,y$ via $G_2\subset \SL(2,\CN)$ and $G_2$ acts on $G_1$ by conjugation. From this action we get a good action of $G$ on $\mc{B}_1$.

Using Proposition~\ref{UniversalIsIntertwiningProp} we get a morphism of deformations $\psi\colon \mc{B}_0\to\mc{B}_1$ that intertwines the action of $G$.


\begin{prop}
\label{InjectivityOfChi12Prop}
There exists a subgroup $H$ of $W$ satisfying the conditions of Corollary~\ref{WeylActuallyActsCor} such that $\psi$ gives an isomorphism between $\mc{B}_0$ and $\mc{B}_1^H$. Moreover, $H$ acts on $\mc{O}_c^2$ and $(\mc{O}_c^2)^H\subset (\mc{O}_c^1)^H$ is a universal commutative deformation of $A^{G_2}\subset A^{G_1}$, where $c$ is the parameter for $\mc{B}_1$.
\end{prop}
\begin{proof}
Define $H$ as follows: $H=\{w\in W \mid wV^G=V^G\}$. For any $w\in H$ we have $w(c)=w(\sum t_i h_i)=\sum t_i w(h_i)$. Since $h_i$ belongs to $V^G$ we get $wh_i\in V^G$, in particular we get $wh_i=\sum_j M_{ij} h_j$. We define the right action of $H$ on $\CN[t_1,\ldots,t_k]$ by $\phi_{w}(t_j)=\sum_i M_{ij} t_i$. This action satisfies $\phi_w(c)=w(c)$, hence the corresponding left action of $H$ satisfies the conditions of Corollary~\ref{WeylActuallyActsCor}.

The action of $G$ is good, $H$ acts by automorphisms of deformations, hence the isomorphism $\chi\colon \mc{B}_1/(t_1,\ldots,t_k)\cong A^{G_1}$ intertwines the action of $G\times H$, where $H$ acts on $A^{G_1}$ trivially. Since the action of $G$ is $\CN[t_1,\ldots,t_k]$-linear, $ghg^{-1}h^{-1}$ is a $\CN[t_1,\ldots,t_k]$-linear map that satisfies $\chi ghg^{-1}h^{-1}=\chi$, in other words $ghg^{-1}h^{-1}$ is a $\CN[t_1,\ldots,t_k]$-linear automorphism of deformations. Using Corollary~\ref{UniqueArrowDeformationCor} with $G_2=G_1$ we get that $ghg^{-1}h^{-1}=id$. It follows that the actions of $G$ and $H$ on $\mc{B}_1$ commute.

The restriction of $\psi$ on $\CN[(V/W)^G]$ corresponds to the natural map $f$ from  $V^G/H$ to $(V/W)^G$. We will prove that  $f$ is an isomorphism. On the level of points $f$ sends an $H$-orbit $O$ to $WO$.

We have a root system in $V$ corresponding to $W$. It gives us a $W$-invariant $\RN$-form of $V$: $V=V_{\RN}+i V_{\RN}$, $WV_{\RN}=V_{\RN}$. Now we define a notion of a dominant element of $V$. Suppose that $x=x_{Re}+x_{Im}\in V$. If $x_{Re}\neq 0$, we say that $x$ is dominant if and only if $x_{Re}$ is dominant. Otherwise we say that $x$ is dominant if and only if $x_{Im}$ is dominant. Consider a $W$-orbit $Wx$. If $(wx)_{Re}=0$ for some $w$, then $(Wx)_{Re}=\{0\}$. It follows that each $W$-orbit contains a unique dominant element.

Let us prove that $f$ is a bijection. Let $O$ be an $W$-orbit such that $gO=O$ for every $g\in G$. 
Consider a unique dominant $x\in O$. Since $G$ acts by automorphisms of Dynkin diagram, $gx$ is also dominant. Hence $gx=x$ for every $g\in G$. This proves the surjectivity of $f$.

Let $\Phi$ be the root system inside $V$ corresponding to $W$. In the case when all $G$-orbits in Dynkin diagram do not contain edges there is a well-known construction of folded root system $\Phi_1$ inside $V^G$. It is defined as follows: $\Phi_1=\{\sum_{g\in G}g\rho\mid\rho\in\Phi\}\setminus\{0\}$. The set of positive roots $\Phi_{1+}$ is defined in the same way with $\Phi_+$ instead of $\Phi$. Denote by $W_1$ the corresponding Weyl group. Let us prove that $H$ contains $W_1$. It is enough to prove that $H$ contains simple reflections. This is clear since for every simple root $\alpha$ of $\Phi_1$ with $\alpha=\sum_{g\in G}g\beta$  we have $s_{\alpha}=\prodl_{\gamma\in G\beta}s_{\gamma}|_{V^G}$.

The only case when $G$-orbit has an edge is the case of $A_{2n}$ Dynkin diagram and $G=C_2$. In this case $\Phi_1=BC_n$. For only simple root $\alpha=\beta_1+\beta_2$ in $BC_n$ with $(\beta_1,\beta_2)\neq 0$ we have $s_{\alpha}=s_{\beta_1+\beta_2}|_{V_G}$. It follows that in this case $H$ also contains $W_1$.

Let us prove that $x\in V^G$ is dominant for $\Phi_1$ if and only if it is dominant for $\Phi$. Indeed, $(x_{Re},\rho)=\frac{1}{|G|}\suml_{g\in G}(g(x_{Re}),\rho)=(x_{Re},\frac{1}{|G|}\suml_{g\in G}g\rho)$, the same for $x_{Im}$. It follows that each $W$-orbit contains no more than one $W_1$-orbit. Hence $f$ is bijective and $H=W_1$.

So $f$ is a bijection between normal algebraic varieties. It follows easily from Zariski Main Theorem that $f$ is an isomorphism.

Therefore $\psi$ gives an isomorphism between $\mc{B}_0$ and $\mc{B}_1^H$. 

From the proof of Proposition~\ref{PropConstructionOfCBHInclusion} we get that $\mc{O}_c^2=(\mc{O}_c^1)^G$. Hence $H$ acts on $\mc{O}_c^2$. Since $\psi$ intertwines the action of $G$ it gives an isomorphism between $\mc{B}_0^G$ and $(\mc{O}_c^2)^H$. It follows that $(\mc{O}_c^2)^H\subset (\mc{O}_c^1)^H$ is a universal commutative deformation of $A^{G_2}\subset A^{G_1}$.
\end{proof}
\section{Descriprion of universal noncommutative deformation.}
\label{NoncommutativeNormalSec}
\subsection{Noncommutative parameter}
\label{NonCommFromPoissonSubsec}
In this section we will classify deformations of $A^{G_2}\subset A^{G_1}$ in the general case.

Let $\chi_{reg}$ be the character of regular representation of $G_1$. Since $\chi_{reg}$ generates the subgroup of imaginary roots inside $Z(\CN[G_1])^*$ we get $W\chi_{reg}=\chi_{reg}$. It follows that the action of $W$ on $Z(\CN[G_1])$ leaves the coefficient on $1$ untouched.

Recall that $H$ is a subgroup of $W$ that acts on $Z(\CN[G_1])\cap Z(\CN[G_2])$. Consider $f=\frac{1}{\lvert H\rvert}h1$. This is an $H$-invariant element with coefficient on $1$ equal to $1$.

Let $R=\CN[z,t_1,\ldots,t_k]$, $c=\sum t_ih_i+zf$.

Let $\mc{O}_c^1$ be a CBH deformation of $A^{G_1}$ with parameter $c$. Define the good action of $G$ on $\mc{O}_c^1$ similarly to the Lemma~\ref{G2G1ActsByDiagramAutomLem}.

Arguing as in the proof of Proposition~\ref{InjectivityOfChi12Prop} we see that $H$ and the CBH parameter $\sum t_ih_i+zf\in Z(R[G_1])$ satisfy the conditions of Corollary~\ref{WeylActuallyActsCor} with the trivial action of $H$ on $z$. We also see that the action of $G$ and $H$ on $\mc{O}_c^1$ commute. Using Proposition~\ref{InjectivityOfChi12Prop} we get the following lemma:

\begin{lem}
\label{LemNonCommCBHDeformation}
Let $R=\CN[z,t_1,\ldots,t_k]$, $c=\sum t_i h_i+zf\in Z(R[G_1])\cap Z(R[G_2])$. Then $H\times G$ acts on $\mc{O}_c^1$. The inclusion $(\mc{O}_c^2)^H\subset (\mc{O}_c^1)^H$ is a deformation of $A^{G_2}\subset A^{G_1}$ such that the base change $z\mapsto 0$ sends this deformation to a universal commutative deformation 
\end{lem}

Now we need several technical statements.

\begin{lem}
\label{PoissonLiftingLem}
Suppose that $K$ is a nontrivial subgroup of $\SL(2,\CN)$, $P\colon\CN[u,v]^K\times\CN[u,v]^K\to\CN[u,v]^K$ is a nonzero bilinear antisymmetric homogeneous mapping of degree $i<0$ satisfying Leibniz identity. Then $i=-2$ and $P$ is proportional to the standard Poisson bracket on $\CN[u,v]^K$.
\end{lem}
\begin{proof}
Proceeding as in~\cite{Etingof-Ginzburg},~Lemma 2.23 we get that $P$ is a restriction of some $K$-equivariant Poisson bracket of degree $i$ on $\CN[u,v]$. Hence $i\geq -2$. If $i=-1$, then $\{u,v\}$ is a $K$-invariant nonzero element of $\CN^2$. There are no such elements for nontrivial $K$.
\end{proof}

\begin{prop}
\label{ZExistsProp}
Suppose that $\mc{A}$ is a deformation of $\CN[u,v]^K$ over $R$. Then there exists an element $z\in R$ of degree $2$ such that $fg-gf+\mc{A}R^{>2}=z\langle f+\mc{A}R^{>0},g+\mc{A}R^{>0}\rangle$, where $\langle\cdot,\cdot\rangle$ is the standard Poisson bracket on $\CN[u,v]$. If $z=0$ then $\mc{A}$ is commutative.
\end{prop}
\begin{proof}
Let $i$ be the smallest nonnegative integer such that $fg-gf+\mc{A}R^{>i}$ is not identically zero (if such $i$ does not exist, we are done with $z=0$). Since $\CN[u,v]^K$ is commutative, $i>0$. The map $(f+\mc{A}R^{>0},g+\mc{A}R^{>0})\mapsto fg-gf+\mc{A}R^{>i}$ is well-defined and satisfies the Leibniz rule.

Take a linear functional $\phi\in (R^i)^*$ such that $\phi(fg-gf+\mc{A}R^{>i})$ is not identically zero. We get a nonzero bilinear homogeneous form of degree $-i$ on $\CN[u,v]^{K}$ satisfying Leibniz rule. The proposition follows easily from Lemma~\ref{PoissonLiftingLem}. 
\end{proof}
\begin{lem}
\label{ZequalsZlem}
Applying this proposition to a deformation $(\mc{O}^1_{\sum t_ih_i+zf})^H$ we get an element $z'$ in $R_0\otimes \CN[z]$. Then $z'=z$.
\end{lem}
\begin{proof}
See, for example, page 15 of~\cite{Loseu239}.
\end{proof}
\subsection{Scheme Y}
This subsection is inspired by Subsections 3.3-3.5 in~\cite{Loseu}.

Denote $A^{G_i}$ by $A_i$.

Let us construct an affine scheme $Y$. It will parametrize deformations of $A_2\subset A_1$ with additional data. Algebra $A_i$ is isomorphic to $\CN[x_i,y_i,z_i]/(f_i(x_i,y_i,z_i))$. Let $D$ be the least common multiple of the degrees of $x_1,y_1,z_1,x_2,y_2,z_2$, $e$ be the maximum of degrees of $f_i$ with respect to $x_i,y_i,z_i$. Let $m=7 D$.

\begin{lem}
\label{LemMultiplicationIsSurjective}
For any $k>0$ we have $(A_1)_{\leq m}^{\cdot k}=(A_1)_{\leq km}$.
\end{lem}
\begin{proof}
It is enough to prove that for $l\geq m$ we have $A_m\cdot A_l=A_{m+l}$. Let $x_1^ay_1^bz_1^c$ be an element of $A_{k+l}$. By definition of $D$ there exist $p,q,r$ such that $\deg (x_1^p)=\deg(y_1^q)=\deg(z_1^r)=D$. Now it is easy to find $a_1p\leq a$, $b_1q\leq b$, $c_1r\leq c$ such that $a_1p\deg x_1+b_1q\deg y_1+c_1r\deg z_1=m$.
\end{proof}

Fix a homogeneous basis $P_1,\ldots,P_N$ of $\CN[x,y]^{G_1}_{\leq me}$ adapted to the flag \[(\CN[x,y]^{G_2})_{\leq me}\subset(\CN[x,y]^{G_1})_{\leq me}.\] We assume that for some $M$ elements $P_1,\ldots,P_M$ form a basis of $\CN[x,y]^{G_2}_{\leq me}$.

\begin{defn}
Suppose that $\mc{A}_2\subset\mc{A}_1$ is a deformation of $A_2\subset A_1$. We say that a sequence of homogeneous elements $a_1,\ldots, a_M\in\mc{A}_2$, $a_{M+1},\ldots,a_N$ is a lift of $P_1,\ldots,P_N$ if the images of $a_1,\ldots,a_M$ in $\CN[x,y]^{G_2}$ coincide with $P_1,\ldots,P_M$ and the images of $a_{M+1},\ldots, a_N$ coincide with $P_{M+1},\ldots,P_N$.
\end{defn}

\begin{stat}
\label{RepresentingObjectForFStat}
There exists a subscheme $Y$ of \[T=\Hom(\bigoplus_{i=1}^e (A_1)_{\leq m}^{\otimes i},(A_1)_{\leq me})\] and a unipotent group scheme $U$ such that \begin{enumerate}
\item
$\CN[Y]$ and $\CN[U]$ are positively graded.
\item
$U$ acts on $Y$ and the corresponding map $\CN[Y]\to \CN[Y]\otimes\CN[U]$ preserves grading.
\item
For any graded algebra $R$ homomorphisms of graded algebras from $\CN[Y]$ to $R$ are in one-to-one correspondence with isomorphism classes of deformations of $i\colon\CN[x,y]^{G_2}\subset\CN[x,y]^{G_1}$ over $R$ with a chosen lift of $P_1,\ldots,P_N$.
\item
$\Hom(\CN[U],R)$-orbits in $\Hom(\CN[Y],R)$ are precisely isomorphism classes of deformations of $A_2\subset A_1$ over $R$.
\end{enumerate}
\end{stat}
\begin{proof}
If $W,V$ are graded finite-dimensional vector spaces, then $\Hom(W,V)$ is naturally graded. This defines a grading on $T$ and $\CN[T]$. Suppose that $\alpha$ is an element of $T$. Then the following are polynomial conditions on $\alpha$ for all $k=1\ldots e$:
\begin{enumerate}
\item
$\alpha(u_1\otimes u_2\otimes\ldots\otimes u_k)=\alpha(\alpha(u_1\otimes u_2\otimes\cdots\otimes u_l)\otimes\alpha(u_{l+1}\otimes\ldots\otimes u_k))$ for all $u_1,u_2,\ldots,u_k$ such that the right-hand side is defined.
\item
$\alpha$ maps $(\CN[x,y]^{G_2}_{\leq m})^{\otimes k}$ to $\CN[x,y]^{G_2}_{\leq mk}$.
\item
$\alpha(u_1\otimes\ldots\otimes u_k)-u_1u_2\ldots u_k$ belongs to $(\CN[x,y]^{G_1})_{<\deg u_1+\ldots+\deg u_k}$ for all homogeneous $u_1,\ldots,u_k$. 
\end{enumerate}
These conditions define a subscheme $\tilde{Y}$. It follows from the third condition that $\CN[\tilde{Y}]$ is positively graded. Suppose that $\alpha$ is a homogeneous $R$-point of $\tilde{Y}$. Denote $(A_1)_{\leq me}$ by $V$. Consider the algebra $\mc{A}=R\otimes T(V)/(\alpha(u_1\otimes u_2\otimes\cdots\otimes u_k)-u_1\otimes u_2\otimes\cdots\otimes u_k)=B\otimes T(V)/I$, where we take all $k=1,\ldots,e$ and $u_1,\ldots,u_k\in (A_1)_{\leq m}$ in the definition of $I$.  We see that $\mc{A}$ is a graded $R$-algebra. We have \begin{multline*}\mc{A}/R^{>0}\mc{A}\cong R\otimes T(V)/(I+R^{>0})=\\
R\otimes T(V)/((\alpha(u_1\otimes u_2\otimes\cdots\otimes u_k)-u_1\otimes u_2\otimes\cdots\otimes u_k), R^{>0})=\\
R\otimes T(V)/(u_1\cdots u_k-u_1\otimes u_2\otimes\cdots\otimes u_k,R^{>0})=T(V)/(u_1\cdots u_k-u_1\otimes\cdots\otimes u_k).
\end{multline*} Here we used the third condition on $\alpha$ to obtain $\alpha(u_1\otimes \cdots\otimes u_k)-u_1\ldots u_k\in R^{>0}\mc{A}$. We get a surjective map from $\mc{A}/R^{>0}\mc{A}$ to $A_1$. Using Lemma~\ref{LemMultiplicationIsSurjective} we see that $\mc{A}/R^{>0}\mc{A}$ is generated by $(A_1)_{\leq m}\subset V$. Hence $\mc{A}/R^{>0}\mc{A}$ is generated by $x_1,y_1,z_1$. Using the third condition on $\alpha$ for $k=2,2,2,e$ we get that $[x_1,y_1]=[y_1,z_1]=[z_1,x_1]=f_1(x_1,y_1,z_1)=0$ in $\mc{A}/R^{>0}\mc{A}$. It follows that $\mc{A}/R^{>0}\mc{A}$ is isomorphic to $A_1$.


The remaining condition on $\mc{A}$ is that $\mc{A}$ should be a free $R$-module. Fix $w_1,w_2,\ldots\in \CN\langle x_1,y_1,z_1\rangle$ such that the images of $w_i$ in $A_1$ form a basis. We see that the images of $w_i$ generate $\mc{A}$ as an $R$-module. Using relations with $[x_1,y_1]$, $[y_1,z_1]$, $[z_1,x_1]$ and $f_1(x_1,y_1,z_1)$ we can express any $w_iw_j$ as a sum $r_{ijk}w_k$, where $r_{ijk}$ depend algebraically on $\alpha$. 

If $\mc{A}$ is a free $R$-module then it coincides with $\oplus Rw_i$ and the multiplication is given by $w_iw_j=\sum r_{ijk} w_k$. In this case $r_{ijk}$ satisfy associativity constraint. The associativity constraint is an algebraic condition on $\alpha$.

On the other hand, suppose that $r_{ijk}$ satisfy associativity constraint. In this case we have an algebra $\mc{A}'=\oplus Rw_i$ with multiplication $w_iw_k=\sum r_{ijk}w_k$ and a surjection from $\mc{A}'$ to $\mc{A}$. 

We note that for any $v\in V$ we have $v=\sum v^i w_i$, where $v^i\in R$ depend algebraically on $\alpha$. Using this we construct an embedding $V\subset \mc{A}'$. We define $\alpha'$ as $\alpha'(v_1\otimes\cdots\otimes v_k)=v_1\ldots v_k\subset V\otimes R\subset \mc{A}'$. The condition $\alpha=\alpha'$ is another algebraic condition on $\alpha$. When this condition is satisfied we have an inverse map from $\mc{A}$ to $\mc{A}'$.

It follows that being a free $R$-module is an algebraic condition on $\alpha$.

Consider the algebra $\mc{A}_2=B\otimes T((A_2)_{\leq me})/(\alpha(u_1\otimes\cdots\otimes u_k)-u_1\otimes \cdots\otimes u_k)$, where we take $u_i$ from $\CN[x,y]^{G_2}$. Using the same argument as above we see that the condition that $\mc{A}_2$ is a deformation of $\CN[x,y]^{G_2}$ is another polynomial condition on $\alpha$.

Let $Y$ be the subscheme of $\tilde{Y}$ defined by these conditions.

The natural homomorphism from $\mc{A}_2$ to $\mc{A}_1$ is a deformation of $A_2\subset A_1$. Consider the images of $P_1,\ldots,P_M$ under the natural map from $(A_2)_{\leq me}\subset V$ to $\mc{A}_2$ and the images of $P_{M+1},\ldots,P_N$ under the natural map from $V$ to $\mc{A}_1$. We obtain a lift of $P_1,\ldots,P_N$ to $\mc{A}_2\subset\mc{A}_1$.

Suppose that $\mc{A}_2\subset\mc{A}_1$ is a deformation of $\CN[x,y]^{G_2}\subset \CN[x,y]^{G_1}$ over $R$, $a_1,\ldots,a_N$ is a lift of $P_1,\ldots,P_N$. We take $V=\Span(a_1,\ldots,a_N)$ and define $\alpha(u_1\otimes\cdots\otimes u_k)$ to be $u_1\ldots u_k\in V\otimes R\subset \mc{A}_1$.

By construction two maps above are inverse to each other.


Let $U$ be the subgroup of $\GL(V)$ consisting of all $\Phi\colon V\to V$ with $\Phi(f)-f\in R[u,v]^{<\deg f}$ for all homogeneous $f$ and $\Phi(f)\in \CN[u,v]^{G_2}$ for all $f\in \CN[u,v]^{G_2}$. The action of a homogeneous $R$-point of $U$ on a homogeneous $R$-point of $Y$ is by conjugation. We see that $U$ is a group scheme, $\CN[U]$ is positively graded, the action of $U$ on $Y$ is algebraic and respects grading and $U(R)$-orbits correspond to isomorphism classes of deformations over $R$.
\end{proof}

Proposition~\ref{ZExistsProp} gives us an element $z_{\alpha}\in R$ for each $\alpha\in \Hom(\CN[Y],R)$. There exists an element $z\in\CN[Y]$ such that $z_{\alpha}=\alpha(z)$ for all $R,\alpha\in \Hom(\CN[Y],R)$: for example, we can take the coefficient on $\{a,b\}$ in $\alpha(a\otimes b - b\otimes a)$ for any $a,b\in (A_1)_{\leq m}$ such that $\{a,b\}\neq 0$. In particular, $\CN[Y]$ is a $\CN[z]$-module.


\subsection{Main theorem}
Recall that we have a chosen basis $P_1,\ldots,P_N$ adapted to the flag $(\CN[u,v]^{G_2})_{\leq me}\subset(\CN[u,v]^{G_1})_{\leq me}$.




Lemma~\ref{LemNonCommCBHDeformation} gives us a deformation $\mc{A}_2\subset\mc{A}_1$ over $\CN[z]\otimes R_0$ that gives a universal commutative deformation when we set $z$ to $0$. Choosing a lift of $P_1,\ldots,P_N$ in $\mc{A}_2\subset\mc{A}_1$ and using Proposition~\ref{RepresentingObjectForFStat} we get a homomorphism of graded algebras from $\CN[Y]$ to $\CN[z]\otimes R_0$. Since $U$ acts on $Y$ we have a homomorphism from $\CN[Y]\to \CN[Y]\otimes \CN[U]$. Combining these two homomorphisms we get a homomorphism $\phi$ from $\CN[Y]$ to $\CN[z]\otimes R_0\otimes\CN[U]$.

Lemma~\ref{ZequalsZlem} tells us that $\phi(z)=z$. Hence $\phi$ is also a homomorphism of $\CN[z]$-modules.

If we specialize $z$ to $0$ we get a homomorphism $\phi_0\colon \CN[Y]/(z)\to R_0\otimes \CN[U]$. The graded algebra $\CN[Y]/(z)$ parametrizes commutative deformations with a chosen lift of $P_1,\ldots,P_N$, the graded algebra $R_0$ parametrizes commutative deformations, therefore $\phi_0$ is isomorphism.

Both $\CN[Y]$ and $\CN[z]\otimes \CN[L]\otimes\CN[U]$ are positively graded $\CN[z]$-modules, $\CN[z]\otimes\CN[L]\otimes\CN[U]$ is a free $\CN[z]$-module, $\phi$ is a homomorphism of graded modules such that $\phi_0$ is an isomorphism. Using graded Nakayama's lemma we see that  $\phi$ is an isomorphism.


\begin{thr}
\begin{enumerate}
\item	
Suppose that $\mc{O}^j$ is the CBH algebra with parameter with parameter \[\sum_{i=1}^m z_ih_i+z_0f\in Z(\CN[z_0,\ldots,z_m][G_j]),\] this is a deformation of $\CN[u,v]^{G_j}$ over $\CN[z_0,\ldots,z_m]$. Then $G\otimes H$ acts on $\mc{O}^1$ and $(\mc{O}^2)^H\subset(\mc{O}^1)^H$ is a universal deformation of $\CN[u,v]^{G_2}\subset\CN[u,v]^{G_1}$.
\item
In case of filtered quantizations every deformation of $i$ is of the form $\mc{O}^2_c\subset\mc{O}^1_c$, where $c\in Z(\CN[G_1])\cap Z(\CN[G_2])$. Parameters $c$ and $c'$ give isomorphic deformations if and only if there exists $w\in H$ such that $c'=wc$.
\end{enumerate}
\end{thr}
\begin{proof}
First statement is clear from the discussion before theorem and the description of $\mc{A}_2\subset\mc{A}_1$ in Lemma~\ref{LemNonCommCBHDeformation}.

Recall that filtered quantization is the same as a deformation over $\CN[z]$. Since homomorphisms of graded algebras from $B$ to $\CN[z]$ are in a natural one-to-one correspondence with $\CN$-points of $B$, the second claim follows from Statement~\ref{CategoricalQuotientStat} applied to $\CN[z_0,\ldots,z_m]^H\subset\CN[z_0,\ldots,z_m]$.
\end{proof}
\bibliography{UniversalArrowsBibliography}
\bibliographystyle{ieeetr}
\textsc{Department of Mathematics, MIT, 77 Mass. Ave, Cambridge, MA 02139}

{\it E-mail address}: \texttt{\href{mailto:klyuev@mit.edu}{klyuev@mit.edu}}
\end{document}